\documentclass[12pt,reqno]{amsart}

\usepackage{amscd}
\usepackage{amssymb}
\usepackage{epsfig}
\usepackage{float}
\usepackage{url}
\usepackage[all]{xy}

\newcommand{\Z}{{\mathbb Z}}
\newcommand{\Q}{{\mathbb Q}}
\newcommand{\C}{{\mathbb C}}
\newcommand{\R}{{\mathbb R}}
\renewcommand{\P}{{\mathbb P}}
\renewcommand{\H}{{\mathbb H}}

\newcommand{\BB}{{\mathcal B}}

\newcommand{\OO}{{\mathcal O}}

\newcommand{\RR}{{\mathcal R}}

\newcommand{\TT}{{\mathcal T}}

\newcommand{\aaa}{{\bf a}}
\newcommand{\ddd}{{\rm d}}
\newcommand{\www}{\widetilde}

\newcommand{\oooo}{\overline}

\newcommand{\iiii}{\infty}
\newcommand{\mmm}{{\bf m}}

\newcommand{\ppp}{\partial}
\newcommand{\mmod}{{\rm mod}}

\DeclareMathOperator{\Aut}{Aut}

\DeclareMathOperator{\id}{id}

\DeclareMathOperator{\mult}{mult}
\DeclareMathOperator{\ord}{ord}
\DeclareMathOperator{\Ord}{Ord}
\DeclareMathOperator{\pr}{pr}

\DeclareMathOperator{\Stab}{Stab}

\DeclareMathOperator{\trace}{trace}

\begin{document}

\theoremstyle{plain}
\newtheorem{lemma}{Lemma}[section]
\newtheorem{definition/lemma}[lemma]{Definition/Lemma}
\newtheorem{theorem}[lemma]{Theorem}
\newtheorem{proposition}[lemma]{Proposition}
\newtheorem{corollary}[lemma]{Corollary}
\newtheorem{conjecture}[lemma]{Conjecture}
\newtheorem{conjectures}[lemma]{Conjectures}

\theoremstyle{definition}
\newtheorem{definition}[lemma]{Definition}
\newtheorem{withouttitle}[lemma]{}
\newtheorem{remark}[lemma]{Remark}
\newtheorem{remarks}[lemma]{Remarks}
\newtheorem{example}[lemma]{Example}
\newtheorem{examples}[lemma]{Examples}

\title
[marked singularities]
{$\mu$-constant monodromy groups\\ and marked singularities}

\author{Claus Hertling}

\address{Universit\"at Mannheim\\ 
Lehrstuhl f\"ur Mathematik VI\\
Seminargeb\"aude A 5, 6\\
68131 Mannheim, Germany}

\email{hertling@math.uni-mannheim.de}

\thanks{This work was supported by the DFG grant He2287/2-1 and the 
ANR grant ANR-08-BLAN-0317-01 (SEDIGA)}



\date{August 2, 2011}

\begin{abstract}
$\mu$-constant families of holomorphic function germs with isolated singularities
are considered from a global perspective.
First, a monodromy group from all families which contain a fixed singularity
is studied. It consists of automorphisms of the Milnor lattice which respect not only
the intersection form, but also the Seifert form and the monodromy.
We conjecture that it contains all such automorphisms, modulo $\pm\id$.
Second, marked singularities are defined and global moduli spaces for right equivalence
classes of them are established. The conjecture on the group would imply that these
moduli spaces are connected. The relation with Torelli type problems is
discussed and a new global Torelli type conjecture for marked singularities is 
formulated. All conjectures are proved for the simple and $22$ of the $28$
exceptional singularities.
\end{abstract}

\maketitle

\tableofcontents

\setcounter{section}{0}

\section{Introduction}\label{c1}
\setcounter{equation}{0}

\noindent
This paper studies local objects from a global perspective.
The local objects are holomorphic function germs $f:(\C^{n+1},0)\to(\C,0)$ with an
isolated singularity at $0$ (short: singularity).
Two types of global objects for them are considered. The first are new monodromy groups,
the {\it $\mu$-constant monodromy groups}.
The second are moduli spaces for {\it marked singularities}. They are related.
And both are important for the study of period maps to spaces of Brieskorn lattices,
that is, regular singular TERP-structures or non-commutative Hodge structures.

The Milnor lattice of a singularity $f$ is $Ml(f):=H_n(f^{-1}(\tau),\Z)\cong \Z^\mu$
(reduced homology if $n=0$), here $\mu$ is the Milnor number, $\tau>0$, 
and $f^{-1}(\tau)$ is a regular fiber in a suitable representative of the function germ $f$.
It comes equipped with two pairings, the intersection form $I$ and the Seifert form $L$,
and with the monodromy $M_h\in\Aut(Ml(f),L,I)$. We put them together in one tuple
$ML(f):=(Ml(f),L,M_h,I)$. In fact, $L$ determines $M_h$ and $I$, so 
$$G_\Z(f):=\Aut(ML(f)) = \Aut(Ml(f),L).$$

We consider two kinds of $\mu$-constant families, either $C^\iiii$-families $F$ of singularities
over a base space $X$ which is a $C^\iiii$-manifold, or holomorphic families $F$ 
where the base space $X$ is a reduced complex space.
In either case the Milnor lattices $Ml(F_t)$, $t\in X$, of the members of the family $F$
glue to a local system of $\Z$-lattices of rank $\mu$ with Seifert form, monodromy
automorphism and intersection form. After fixing one point $t_0\in X$, the monodromy
group $G(F,X,t_0)$ of such a family is the image of the natural homomorphism
$\pi_1(X,t_0)\to G_\Z(F_{t_0})$ (definition \ref{t3.1} (a)).

For a singularity $f$, the {\it $\mu$-constant monodromy group} $G^{smar}(f)$ is the
subgroup of $G_\Z(f)$ generated by all monodromy groups of all $\mu$-constant
families which contain $f$ (definition \ref{t3.1} (b)). 
But using $k$-jets and the finite determinacy of singularities,
it is not hard to construct one global holomorphic $\mu$-constant family whose 
monodromy group is $G^{smar}(f)$ (lemma \ref{t3.5} (c)).

This global $\mu$-constant family was the starting point in \cite[theorem 13.15]{He6}
for the construction of a global moduli space $M_\mu(f_0)$ for right equivalence classes
of singularities in the $\mu$-homotopy class of a fixed singularity $f_0$. 
Here we will adapt this construction and establish a moduli space $M_\mu^{mar}(f_0)$
[respectively $M_\mu^{smar}(f_0)$] of [strongly] marked singularities (theorem \ref{t4.3}).

Fix one singularity $f_0$. A {\it [strongly] marked singularity} is a pair $(f,\pm \rho)$
[respectively $(f,\rho)$] where $f$ is a singularity in the $\mu$-homotopy class
of $f_0$ and $\rho:ML(f)\to ML(f_0)$ is an isomorphism.
Here $\pm \rho$ means the set $\{\rho,-\rho\}$, so neither $\rho$ nor $-\rho$ is 
distinguished. Two [strongly] marked singularities $(f_1,\pm\rho_1)$ and $(f_2,\pm\rho_2)$
[$(f_1,\rho_1)$ and $(f_2,\rho_2)$] are {\it right equivalent} if a 
coordinate change $\varphi:(\C^{n+1},0)\to (\C^{n+1},0)$ exists with
$f_1=f_2\circ \varphi$ and $\rho_1=\pm\rho_2\circ\varphi_{hom}$
[respectively $\rho_1=\rho_2\circ\varphi_{hom}$], here $\varphi_{hom}:ML(f_1)\to ML(f_2)$
is the induced isomorphism.

A surprising fact is that the strongly marked singularities $(f,\rho)$ and $(f,-\rho)$
are right equivalent if and only if $\mult f=2$
(theorem \ref{t3.3} (e) and (g)).
This leads to potential problems for the space $M_\mu^{smar}$: If there would exist
a $\mu$-homotopy class which contains singularities with multiplicity $2$ and 
singularities with multiplicity $\geq 3$ (which I don't believe), then its 
moduli space $M_\mu^{smar}$ of strongly marked singularities would not be Hausdorff
(theorem \ref{t4.3} (e)).
The moduli space $M_\mu^{mar}$ is not sensitive to this, it exists always
as an analytic geometric quotient.

This is one reason why we consider not only strongly marked singularities, but also
marked singularities. The other is that the period map $M_\mu^{smar}(f_0)\to D_{BL}(f_0)$
to a classifying space for Brieskorn lattices factors through $M_\mu^{mar}(f_0)$.

Locally $M_\mu^{mar}$ and $M_\mu^{smar}$ (if it is Hausdorff) are isomorphic to 
the $\mu$-constant stratum of a singularity (theorem \ref{t4.3} (b)).
The group $G_\Z(f_0)$ acts on $M_\mu^{mar}$ and $M_\mu^{smar}$ by
$\psi:[(f,\pm\rho)]\mapsto [(f,\pm\psi\circ\rho)]$ 
[respectively $\psi:[(f,\rho)]\mapsto[(f,\psi\circ\rho)]$].
The action is properly discontinuous (on $M_\mu^{smar}$ if it is Hausdorff),
the quotient is $M_\mu$ (theorem \ref{t4.3} (d)).

The $\mu$-constant monodromy group $G^{smar}(f_0)$ turns out to be the subgroup of $G_\Z(f_0)$
which acts on that topological component $(M_\mu^{smar})^0$ which contains $[(f_0,\id)]$
(if $M_\mu^{smar}$ is Hausdorff), and likewise, 
the group 
$$G^{mar}(f_0):=\{\pm \psi\, |\, \psi\in G^{smar}(f_0)\}$$
is the subgroup of $G_\Z(f_0)$ which acts on the component $(M_\mu^{mar})^0$
which contains $[(f_0,\pm\id)]$ (theorem \ref{t4.4} (a) and (b)).
Especially, there is a 1--1 correspondence
\begin{eqnarray*}
G_\Z(f_0)/G^{mar}(f_0)&\to & \{\textup{topological components of }M_\mu^{mar}\}\\
\psi \cdot G^{mar}(f_0)&\mapsto &\psi ((M_\mu^{mar})^0)
\end{eqnarray*}
[and similarly for $G^{smar}(f_0)$ and $M_\mu^{smar}$ if it is Hausdorff].

\begin{conjecture}(=\ref{t3.2})
(a) $G^{mar}(f_0)=G_\Z(f_0)$, equivalent: $M_\mu^{mar}$ is connected.

(b) If all singularities in the $\mu$-homotopy class of $f_0$ have multiplicity $\geq 3$
then $G^{mar}(f_0)=G^{smar}(f_0)\times\{\pm\id\}$, equivalent:
$-\id\notin G^{smar}(f_0)$.
\end{conjecture}

Part (a) is a fragile conjecture. If it is true, it points at hidden properties which
distinguish the lattice $Ml(f_0)$ from other monodromy invariant lattices in $
Ml(f_0)\otimes_\Z\Q$. For example, it implies that any basis which has the same 
Coxeter-Dynkin diagram as a distinguished basis is also distinguished (remark \ref{t3.4}).

Part (b) leads to the question how in cases where it is true, the index $2$ subgroup
$G^{smar}(f_0)\subset G^{mar}(f_0)$ can be described a priori.
Both conjectures are proved in section \ref{c8} for the simple and $22$ of the
$28$ families of exceptional singularities. In another paper they will be proved for the 
remaining 6 families of exceptional singularities, for the 
simple-elliptic and the hyperbolic singularities.

In \cite{He4} a classifying space $D_{BL}(f_0)$ for (candidates of) Brieskorn lattices
was constructed. It is a complex manifold, and $G_\Z(f_0)$ acts properly discontinuously
on it. Now one obtains a holomorphic period map
$$BL:M_\mu^{mar}(f_0)\to D_{BL}(f_0),$$
which is $G_\Z(f_0)$-equivariant.
An infinitesimal Torelly type result is that it is an immersion (\cite[theorem 12.8]{He6}).
The following is a global Torelli type conjecture for marked singularities.

\begin{conjecture}(=\ref{t5.3})
$BL:M_\mu^{mar}(f_0)\to D_{BL}(f_0)$ is injective.
\end{conjecture}

It is equivalent to two Torelli type conjectures which I had proposed earlier.
One is that the period map after taking the quotient by $G_\Z(f_0)$ 
$$LBL:M_\mu^{mar}(f_0)/G_\Z(f_0)=M_\mu(f_0)\to D_{BL}(f_0)/G_\Z(f_0)$$
is injective. It says that the right equivalence class of a singularity is determined
by its Brieskorn lattice (up to isomorphism). I worked on it in \cite{He1}--\cite{He6}.
The other is that for any $[(f,\pm \rho)]\in M_\mu^{mar}(f_0)$
$$\Stab_{G_\Z(f_0)}([(f,\pm\rho)]) = \Stab_{G_\Z(f_0)}(BL([(f,\pm\rho)])),$$
this is \cite[conjecture 13.12]{He6}. Obvious is only $\subset$ and that both groups
are finite, because $G_\Z(f_0)$ acts properly discontinuously on $M_\mu^{mar}(f_0)$
and $D_{BL}(f_0)$.

Nevertheless, the isotropy group $\Stab_{G_\Z(f_0)}([(f,\pm\rho)])$ and also the subgroup
$\Stab_{G_\Z(f_0)}([(f,\rho)]) $ are much better understood than the monodromy groups
$G^{mar}(f_0)$ and $G^{smar}(f_0)$. The isotropy groups had been studied from the 
point of view of symmetries of singularities in \cite[13.1 and 13.2]{He6}.
Section \ref{c6} reviews the results.

The isotropy group $\Stab_{G_\Z(f_0)}([(f,\rho)])$ can also be seen as a $\mu$-constant
monodromy group, but for $\mu$-constant families where all members are right
equivalent to $f$ (theorem \ref{t4.4} (d)).

This paper deals almost exclusively with $\mu$-constant families of singularities.
Semiuniversal unfoldings are only used in the discussion of symmetries of singularities
and in the construction of $M_\mu^{mar}$. But later I hope to extend $M_\mu^{mar}$
to a manifold of dimension $\mu$ which is locally a semiuniversal unfolding
and which allows to consider distinguished bases and Stokes data of 
deformations which are not $\mu$-constant from a global perspective.

Section \ref{c2} reviews the topology of singularities.
Section \ref{c3} defines and studies the $\mu$-constant monodromy groups.
Section \ref{c4} establishes the moduli spaces for [strongly] marked singularities,
though the main proof is given in section \ref{c7}. Section \ref{c5} discusses
the period maps $BL$ and $LBL$. Section \ref{c6} reviews the symmetries of 
singularities. Section \ref{c8} proves all conjectures for the simple and $22$
of the $28$ exceptional singularities.

I thank the organizers for the workshop on the geometry and physics of the
Landau-Ginzburg model in the summer 2010 in Grenoble.
I thank Martin Guest and the Tokyo Metropolitan University for hospitality
in the winter 2010/2011.

\section{Review on the topology of isolated hypersurface singularities}\label{c2}
\setcounter{equation}{0}

\noindent
First, we recall some classical facts and fix some notations.
An {\it isolated hypersurface singularity} (short: {\it singularity})
is a holomorphic function germ $f:(\C^{n+1},0)\to (\C,0)$ with an isolated 
singularity at $0$. Its {\it Milnor number} 
$$\mu:=\dim\OO_{\C^{n+1},0}/(\frac{\ppp f}{\ppp x_i})$$ 
is finite.
A {\it Milnor fibration} for $f$ is constructed as follows \cite{Mi}.
Choose $\varepsilon>0$ such that $f$ is defined on the ball
$B^{2n+2}_{\varepsilon}:=\{x\in \C^{n+1}\, |\, |x|<\varepsilon\}$
and $f^{-1}(0)$ is transversal to $\ppp B^{2n+2}_{\www\varepsilon}$ 
for all $\www\varepsilon\leq \varepsilon$. 
Choose $\delta>0$ such that $f^{-1}(\tau)$ is transversal to
$\partial B^{2n+2}_\varepsilon$ for all 
$\tau\in T_\delta:=\{\tau\in\C\, |\, |\tau|<\delta\}$.
Define $T_\delta':=T_\delta-\{0\}$,  
$Y(\varepsilon,\delta):=B^{2n+2}_\varepsilon\cap f^{-1}(T_\delta)$
and $Y'(\varepsilon,\delta):=Y(\varepsilon,\delta)-f^{-1}(0)$.
Then $f:Y'(\varepsilon,\delta)\to T_\delta'$ is a locally trivial 
$C^\infty$-fibration, the {\it Milnor fibration}, and each fiber has the
homotopy type of a bouquet of $\mu$ $n$-spheres \cite{Mi}.

Therefore the (reduced for $n=0$) middle homology groups are {}\\{}
$H_n^{(red)}(f^{-1}(\tau),\Z) \cong \Z^\mu$ for $\tau\in T_\delta'$.
Each comes equipped with an intersection form $I$, which is a datum of one fiber,
a monodromy $M_h$ and a Seifert form $L$, which come from the Milnor fibration,
see \cite[I.2.3]{AGV2} for their definitions (for the Seifert form, there are several
conventions in the literature, we follow \cite{AGV2}). 
$M_h$ is a quasiunipotent automorphism, $I$ and $L$ are bilinear forms with values in $\Z$,
$I$ is $(-1)^n$-symmetric and $L$ is unimodular. $
L$ determines $M_h$ and $I$ because of the formulas
\cite[I.2.3]{AGV2}
$$L(M_ha,b)=(-1)^{n+1}L(b,a),\quad
I(a,b)=-L(a,b)+(-1)^{n+1}L(b,a).$$
If $f:Y'(\www\varepsilon,\www\delta)\to T_{\www\delta}'$ is a Milnor fibration 
with $\www\varepsilon<\varepsilon$ and $\www\delta<\delta$ then the inclusion
$$Y'(\www\varepsilon,\www\delta)\cap f^{-1}(\ppp T_{\www\delta}')\hookrightarrow
Y'(\varepsilon,\delta)\cap f^{-1}(\ppp T_{\www\delta}')$$
is a fiber homotopy equivalence between the restrictions to 
$\ppp T_{\www\delta}'$ of the new and the old Milnor fibration 
(\cite{Mi} or \cite[Lemma 2.2]{LR}).
Therefore the Milnor lattices $H_n(f^{-1}(\tau),\Z)$ for all Milnor fibrations
and all $\tau\in\R_{>0}\cap T_{\delta}'$ are canonically isomorphic,
and the isomorphisms respect $M_h$, $I$ and $L$. These lattices are identified
and called $Ml(f)$, the tuple $(Ml(f),L,M_h,I)$ is called $ML(f)$
(for MiLnor and Lattice and $L=$ Seifert form). 
Remark that $\Aut(ML(f))=\Aut(Ml(f),L)$ because $L$ determines $M_h$ and $I$.
This group is called $G_\Z(f)$.

The function germ $f(x_0,...,x_n)+x_{n+1}^2\in \OO_{\C^{n+2},0}$ is called 
{\it stabilization} or {\it suspension} of $f$. There is a canonical
isomorphism $Ml(f)\otimes Ml(x_{n+1}^2)\to Ml(f+x_{n+1}^2)$ \cite[I.2.7]{AGV2}.
As there are only two isomorphisms $Ml(x_{n+1}^2)\to\Z$, and they differ by a sign,
there are two equally canonical isomorphisms $Ml(f)\to Ml(f+x_{n+1}^2)$,
and they differ just by a sign. 
Therefore automorphisms and bilinear forms on $Ml(f)$ can be identified with 
automorphisms and bilinear forms on $Ml(f+x_{n+1}^2)$. In this sense
$$L(f+x_{n+1}^2) = (-1)^n\cdot L(f)\qquad\textup{and}\qquad 
M_h(f+x_{n+1}^2)= - M_h(f)$$
\cite[I.2.7]{AGV2}, and $G_\Z(f+x_{n+1}^2)=G_\Z(f)$.

The group of biholomorphic map germs $\varphi:(\C^{n+1},0)\to(\C^{n+1},0)$
is called $\RR$, its elements are called {\it coordinate changes}.
Two singularities $f$ and $g\in \OO_{\C^{n+1},0}$ are {\it right equivalent},
if $f=g\circ \varphi$ for some $\varphi\in\RR$, notation: $f\sim_\RR g$.
In that case $\varphi$ induces an isomorphism
$$\varphi_{hom}: ML(f)\to ML(g).$$

The {\it multiplicity} of $f$ is $\mult f:=\max(k\, |\, f\in \mmm^k)$,
here $\mmm\subset\OO_{\C^{n+1},0}$ is the maximal ideal.
The {\it splitting lemma} says for isolated hypersurface singularities
$f,f_1,f_2\in\mmm^2_{\C^{n+1},0}$ \cite{AGV1}
\begin{eqnarray*}
\mult f=2&\iff& f\sim_\RR g(x_0,...,x_{n-1})+x_n^2\quad\textup{for some}
\quad g\in \mmm^2_{\C^n,0}\\
f_1\sim_\RR f_2 &\iff& f_1+x_{n+1}^2\sim_\RR f_2+x_{n+1}^2
\end{eqnarray*}
(in the first equivalence $\Leftarrow$ is trivial, in the second $\Rightarrow$).

The next definition and the theorem after it are preparations for section \ref{c3}.

\begin{definition}\label{t2.1}
(a) A {\rm $C^\infty$ $\mu$-constant family} consists of a number $\mu\in \Z_{\geq 1}$,
a connected $C^\iiii$-manifold $X$, possibly with boundary (e.g. $X=[0,1]$),
an open neighborhood $Y\subset \C^{n+1}\times X$ of $\{0\}\times X$ 
and a $C^\iiii$-function $F:Y\to\C$, such that $F_t:=F_{|Y\cap \C^{n+1}\times\{t\}}$
for any $t\in X$ is holomorphic and has an isolated singularity with Milnor number 
$\mu$ at 0.

(b) A {\rm holomorphic $\mu$-constant family} consists of a number $\mu\in \Z_{\geq 1}$,
a connected reduced complex space $X$,
an open neighborhood $Y\subset \C^{n+1}\times X$ of $\{0\}\times X$ 
and a holomorphic function $F:Y\to\C$, such that $F_t:=F_{|Y\cap \C^{n+1}\times\{t\}}$
for any $t\in X$ has an isolated singularity with Milnor number 
$\mu$ at 0.

(c) The {\rm $\mu$-homotopy class} of $f$ consists of all singularities $g$ such that
a $C^\iiii$ $\mu$-constant family exists which contains $f$ and $g$.
\end{definition}

\begin{theorem}\label{t2.2}
In both cases ((a) and (b) in definition \ref{t2.1}) 
the Milnor lattices $Ml(F_t)$ and the tuples $ML(F_t)$ for $t\in X$
are locally canonically isomorphic. They glue to a local system $Ml(F)$ of 
free $\Z$-modules of rank $\mu$ on $X$ with a flat unimodular pairing $L$,
a flat automorphism $M_h$ and a flat intersection form $I$.
The tuple $(Ml(F),L,M_h,I)$ is called $ML(F)$.
\end{theorem}

\begin{proof}
(a) For any $t\in X$ one can choose $\varepsilon(t)$ and $\delta(t)$ such that
$F_t:Y'(\varepsilon(t),\delta(t))\times\{t\}\to T_{\delta(t)}'$ is a Milnor fibration.
But it may happen that $\epsilon(t)$ and $\delta(t)$ cannot be chosen as
continuous functions (a {\it vanishing fold} might exist).
Luckily \cite[Lemma 2.2]{LR} says that for $F_t$ with $t$ close to $t_0$ and 
$\varepsilon(t)\leq \varepsilon(t_0)$, $\delta(t)\leq \delta(t_0)$,
the inclusion 
$$Y'(\varepsilon(t),\delta(t))\times\{t\}\cap F_t^{-1}(\ppp T_{\delta(t)}')
\hookrightarrow 
B^{2n+2}_{\epsilon(t_0)}\times\{t\}\cap F_t^{-1}(\ppp T_{\delta(t)}')$$
is a fiber homotopy equivalence over $\ppp T_{\delta(t)}'$.
And the second fibration is obviously diffeomorphic to the restriction of the
Milnor fibration of $F_{t_0}$ to $\ppp T_{\delta(t)}'$. This proves (a).

(b) This follows from (a).
\end{proof}

\section{$\mu$-constant monodromy groups}\label{c3}
\setcounter{equation}{0}

Definition \ref{t3.1} presents the first main subject of this paper,
the $\mu$-constant monodromy groups and some subgroups.

\begin{definition}\label{t3.1}
Let $f\in \mmm_{\C^{n+1},0}^2$ have an isolated singularity at $0$.

(a) For any $C^\infty$ or holomorphic $\mu$-constant family $(X,Y,F)$ 
(definition \ref{t2.1}) 
with $F_{t_0}=f$ for some $t_0\in X$, the local system $Ml(F)$ over $X$ yields
a homomorphism $\pi_1(X,t_0)\to G_\Z(f)$. The image is the 
{\rm $\mu$-constant monodromy group} $G(F,t_0)\subset G_\Z(f)$ of this 
$\mu$-constant family.

If $X=S^1$ we call the image of the standard generator of $\pi_1(S^1,t_0)$
the {\rm monodromy} of the $\mu$-constant family.

(b) We define four subgroups of $G_\Z(f)$. 
The first two are called {\rm $\mu$-constant monodromy groups} of $f$.
\begin{eqnarray*}
G^{smar}(f)&:=&\{\textup{the subgroup generated by all }G(F,t_0)
\textup{ as in (a)}\},\\
G^{mar}(f)&:=&\{\pm \psi\, |\, \psi\in G^{smar}(f)\},\\
G^{smar}_\RR(f)&:=&\{\textup{the subgroup generated by all }G(F,t_0)
\textup{ as in (a)} \\
 &&  \textup{ \ where }F_t\sim_\RR f\textup{ for all }t\in X\},\\
G^{mar}_\RR(f)&:=&\{\pm \psi\, |\, \psi\in G^{smar}_\RR(f)\}.
\end{eqnarray*}
\end{definition}

In lemma \ref{t3.5} (c) and in theorem \ref{t3.3} (e) other more compact 
descriptions of $G^{smar}(f)$ and $G^{smar}_\RR(f)$ will be given. 
The indices "{}smar{}" and "{}mar{}" stand for {\it strongly marked}
and {\it marked}. They are motivated by theorem \ref{t4.4} (a) and (b). 
Theorem \ref{t4.4} will put $G^{smar}(f)$ and $G^{mar}(f)$ into action.
Obviously
\begin{eqnarray*}
\begin{matrix}G^{smar}_\RR(f)&\subset& G^{mar}_\RR(f) & & \\
 \cap & & \cap & & \\
G^{smar}(f)&\subset&
G^{mar}(f)&\subset& G_\Z(f).\end{matrix}
\end{eqnarray*}

The two groups $G^{smar}_\RR(f)$ and $G^{mar}_\RR(f)$ are finite (theorem \ref{t6.1} (f)).
They depend on the right equivalence class of $f$. They were studied already
in \cite[ch. 13]{He6}. We cite some results about them in theorem \ref{t3.3} and discuss
them in section \ref{c6}. 
Conjecture \ref{t5.1} would give complete control on them through the Brieskorn lattice.

The two groups $G^{smar}(f)$ and $G^{mar}(f)$ depend up to conjugacy only on the 
$\mu$-homotopy class of $f$. They are hard to calculate.
I propose the following two conjectures.

\begin{conjecture}\label{t3.2}
Let $f\in \mmm^2_{\C^{n+1},0}$ have an isolated singularity at $0$.

(a) $$G^{mar}(f)=G_\Z(f).$$
(b) If all singularities in the $\mu$-homotopy class of $f$ have multiplicity $\geq 3$
then $-\id \notin G^{smar}(f)$, equivalent: then $G^{mar}(f)=G^{smar}(f)\times\{\pm\id\}$.
\end{conjecture}

At first sight, conjecture \ref{t3.2} (a) might look safe as all monodromy groups 
of all $\mu$-constant families together should give a large subgroup of $G_\Z(f)$.
At second sight, it turns out to be a fragile conjecture.
Often there are other $\Z$-lattices $V_\Z$ of maximal rank in $Ml(f)\otimes_\Z\Q$
such that $\Aut(V_\Z,L)\supsetneq G_\Z(f)$. Conjecture \ref{t3.2} (a) is related 
to hidden properties which distinguish $(Ml(f),L)$ from other $\Z$-lattices in
$Ml(f)\otimes_\Z\Q$. Section \ref{c8} will give examples.
See also remark \ref{t3.4}.

Conjecture \ref{t3.2} (b) is even more mysterious. If both conjectures are true
then $G_\Z(f)=G^{smar}(f)\times\{\pm \id\}$ for $f$ as in (b).
Is there an a priori way to distinguish such a subgroup of index $2$ in $G_\Z(f)$?

Theorem \ref{t3.3} collects some evidence for the conjectures and some results
about the four groups.
The singularities with modality $\leq 2$ are given in \cite{AGV1}.

\begin{theorem}\label{t3.3}
Let $f\in \mmm^2_{\C^{n+1},0}$ have an isolated singularity at $0$.

(a) The conjectures \ref{t3.2} (a) and (b) are true for all singularities
with modality $\leq 1$, that means, simple ($ADE$), simple-elliptic
(=parabolic, $\www E_6=P_8$, $\www E_7=X_9$, $\www E_8=J_{10}$), 
hyperbolic ($T_{pqr}$) and exceptional unimodal.
They are also true for the 14 families of exceptional bimodal singularities
(for the other bimodal singularities I did not yet make enough calculations),
and for the Brieskorn-Pham singularities $\sum_{i=0}^n x_i^{a_i}$ 
with pairwise coprime exponents.

(b) If some singularity in the $\mu$-homotopy class of $f$ has multiplicity $2$
then $-\id\in G^{smar}(f)$, equivalent: then $G^{smar}(f)=G^{mar}(f)$.

(c) $M_h\in G^{smar}(f)$. If $f$ is quasihomogeneous then $M_h\in G^{smar}_\RR(f)$.

(d) If $\mult f\geq 3$ then $M_h^k\neq-\id$ for any $k\in \Z$.

(e) $$G^{smar}_\RR(f)=\{\varphi_{hom}\in G_\Z(f)\, |\, \varphi\in \RR\textup{ with }
f=f\circ \varphi\}.$$

(f) $G^{smar}_\RR(f)$ and $G^{mar}_\RR(f)$ are finite.

(g) $$-\id\notin G^{smar}_\RR(f)\iff \mult f\geq 3.$$
Equivalent: $G^{mar}_\RR(f)=G^{smar}_\RR(f)$ if $\mult f=2$, and
$G^{mar}_\RR(f)=G^{smar}_\RR(f)\times\{\pm\id\}$ if $\mult f\geq 3$.

(h) $G^{mar}_\RR(f)=G^{mar}_\RR(f+x_{n+1}^2)$.

(i) If all singularities in the $\mu$-homotopy class of $f+x_{n+1}^2$ have
multiplicity $2$ then $G^{mar}(f)=G^{mar}(f+x_{n+1}^2)$.
\end{theorem}

\begin{proof}
(a) See theorem \ref{t8.3}, theorem \ref{t8.4} and remark \ref{t8.5} (i) for
the simple, 22 of the 28 exceptional and the Brieskorn-Pham singularities. 
The remaining 6 families of unimodal and bimodal exceptional singularities,
the simple-elliptic and the hyperbolic singularities will be treated in another paper.

(b) Let $g(x_0,...,x_{n-1})+x_n^2$ be in the $\mu$-homotopy class of $f$.
Then 
$$g+t\cdot x_n^2,\quad t\in S^1=:X,$$
is a $C^\iiii$ $\mu$-constant family. With the canonical isomorphism
$Ml(g+x_n^2)\cong Ml(g)\otimes Ml(x_n^2)$ one sees that its monodromy is 
$\id\otimes (-\id)=-\id\in G_\Z(g+x_n^2)$. Any path from $f$ to $g+x_n^2$ in a 
$C^\iiii$ $\mu$-constant family induces an isomorphism $B:ML(f)\to ML(g+x_n^2)$
with $G^{smar}(f)=B^{-1}\circ G^{smar}(g+x_n^2)\circ B$.

(c) The monodromy of the $C^\iiii$ $\mu$-constant family
$$ F_t:= t^{-1}\cdot f,\quad t\in S^1,$$
is $M_h$, because $F_t^{-1}(\tau)=f^{-1}(t\cdot \tau)$ for $\tau>0$.
If $f$ is quasihomogeneous then $F_t\sim_\RR f$.

(d) Th\'eor\`eme 1 in \cite{AC} (and already a letter from Deligne to A'Campo,
see \cite{AC}) shows
$$\trace M_h^k=(-1)^{n+1}\qquad \textup{if }0<k<\mult f.$$
Let $\Phi_m$ for $m\in \Z_{\geq 1}$ be the cyclotomic polynomial of primitive
unit roots of order $m$. It is well known that
\begin{eqnarray*}
\sum_{\Phi_m(\lambda)=0}\lambda = \left\{\begin{matrix}
0 & \textup{ if }p^2|m\textup{ for some prime number }p,\\
(-1)^s & \textup{ if }m=p_1\cdot ...\cdot p_s\textup{ with different}\\
 & \textup{prime numbers }p_1,...,p_s.\end{matrix}\right.
\end{eqnarray*}
Now suppose $\mult f\geq 3$. Then $\mu\geq 2$ and $\trace M_h=\trace M_h^2=(-1)^{n+1}$.
Further suppose $M_h^k=-\id$ for some $k\in\Z$. Then there exist
$a\in \Z_{\geq 1}$ and odd numbers $b_1,...,b_r\in \Z_{\geq 1}$ with
$$\{\ord \lambda\, |\, \lambda \textup{ eigenvalue of }M_h\}
=\{2^a\cdot b_1,...,2^a\cdot b_r\}.$$
Now $\trace M_h\neq 0$ implies $a=1$. 
Then $\{\ord\lambda\, |\, \lambda\textup{ eigenvalue of }M_h^2\}=\{b_1,...,b_r\}$
and $\trace M_h^2=-\trace M_h$, a contradiction.

(e) See theorem \ref{t4.4} (d). 

(f)+(g)+(h) See theorem \ref{t6.1} (i)

(i) See theorem \ref{t4.4} (e).
\end{proof}

\begin{remarks}\label{t3.4}
There is a set $\BB^*\subset Ml(f)^\mu$ of {\it distinguished bases},
see \cite{AGV2} or \cite{Eb} for the definition. 

\noindent
{\bf Claim:} The elements of $G^{mar}(f)$
respect this set, so
$$G^{mar}(f)\subset \Aut(Ml(f),L,\BB^*)\subset G_\Z(f).$$
If one knows how distinguished bases arise, it is not hard to see this claim.
I will discuss it in another paper. Here I just want to point to an implication of 
conjecture \ref{t3.2} (a): It would imply equalities. Equivalent to the second equality 
$\Aut(Ml(f),L,\BB^*)= G_\Z(f)$ is that any basis of $Ml(f)$ which has
the same Coxeter-Dynkin diagram as some distinguished basis is also distinguished.
This is true for the singularities in theorem \ref{t3.3} (a). For the simple
and the simple-elliptic singularities there are older proofs.
It seems to be hard to establish it in any case.
\end{remarks}

Now we will describe a holomorphic $\mu$-constant family which in a certain sense
induces any $\mu$-constant family of singularities in a fixed $\mu$-homotopy class
and whose monodromy group is the group $G^{smar}$.
This is based on the theory of Tougeron and Mather of jets and finite determinacy
of singularities \cite{Ma2} (see also \cite{BL}).

Write $\OO=\OO_{\C^{n+1},0}$ and $\mmm=\mmm_{\C^{n+1},0}$. The $k$-jet of a function
germ $f\in\OO$ is the class $j_kf\in\OO/\mmm^{k+1}$, the $k$-jet of a 
coordinate change $\varphi=(\varphi_0,...,\varphi_n)\in\RR$ is 
$j_k\varphi=(j_k\varphi_0,...,j_k\varphi_n)$. 
The action of $\RR$ on $\mmm^2$ pushes down to an action of the algebraic group
$j_k\RR$ on $\mmm^2/\mmm^{k+1}$ as a smooth affine algebraic variety.

By a result of Tougeron and Mather \cite[theorem (3.5)]{Ma2} a singularity
$f\in \mmm^2$ with Milnor number $\mu$ is $\mu+1$-determined, that means,
any function germ $g\in\mmm^2$ with $j_{\mu+1}g=j_{\mu+1}f$ is right equivalent
to $f$.

Fix $\mu$ and $k\geq \mu+1$. For any singularity $g\in\mmm^2$ with $\mu(g)\leq k-1$
the codimension of the orbit $j_k\RR\cdot j_kg$ in $\mmm^2/\mmm^{k+1}$
is $\mu(g)-1$. The union of all orbits with codimension $\geq \mu-1$ is an algebraic
subvariety of $\mmm^2/\mmm^{k+1}$. 
The set $\{j_kg\, |\, g\in \mmm^2,\mu(g)=\mu\}$ is Zariski open in it and thus
a quasiaffine variety.

For a fixed singularity $f\in\mmm^2$ with $\mu(f)=\mu$ denote by $C(k,f)$ the 
topological component of it which contains $j_kf$. It is also a quasiaffine variety.
For any $t\in C(k,f)$ denote by $F_t$ the unique polynomial of degree $\leq k$
with $j_kF_t=t$. These polynomials glue to a regular function 
$F:\C^{n+1}\times C(k,f)\to\C$.

\begin{lemma}\label{t3.5}
(a) $(C(k,f),\C^{n+1}\times C(k,f),F)$ is a holomorphic $\mu$-constant family.

(b) For any $C^\infty$ $\mu$-constant family $(X,Y,E)$ in the $\mu$-homotopy class of $f$
the local system $ML(E)$ over $X$
is obtained from $ML(F)$ by pull back via the jet map 
$X\to C(k,f),\ t\mapsto j_kE_t$.

(c) Denote by $B:ML(f)\to ML(j_kf)$ the isomorphism induced from the $C^\iiii$
$\mu$-constant family $f+t\cdot (j_kf-f)$, $t\in [0,1]$. Then
\begin{eqnarray*}
G^{smar}(f) &=& B^{-1}\circ G(F,j_kf)\circ B,\\
G^{smar}_\RR(f) &=& B^{-1}\circ G(F_{|j_k\RR\cdot j_kf},j_kf)\circ B.
\end{eqnarray*}
\end{lemma}

\begin{proof}
(a) $\mu(F_t)=\mu$ due to the finite determinacy.

(b) Again due to the finite determinacy, the family
$$\www E_{(s,t)} := E_s +t\cdot (j_kE_s - E_s),\quad (s,t)\in X\times [0,1],$$
is a $C^\infty$ $\mu$-constant family. Its restriction to $X\times \{1\}$ 
is induced by $F$ via the natural map $X\times\{1\}\to C(k,f)$.

(c) This follows from (b) and an analogous statement for $\mu$-constant families
$(X,Y,G)$ with $G_s\sim_\RR f$ for any $s\in X$ and the restriction of the family
$F$ to $j_k\RR\cdot j_kf$.
\end{proof}

\begin{remarks}\label{t3.6}
(i) In definition \ref{t3.1} (b) it is sufficient to consider $C^\infty$ 
$\mu$-constant families over $S^1$. 
Replacing $C^\iiii$ by continuous and piecewise $C^\iiii$ would not give more,
because one can smoothen a curve at a point where it is only piecewise $C^\infty$
by a reparametrization using a (monotonous) $C^\iiii$-function
$[0,1]\to[0,1]$ with $0\mapsto 0$, $\frac{1}{2}\mapsto \frac{1}{2}$ 
and $1\mapsto 1$ which is constant near $\frac{1}{2}$.

(ii) Also a $\mu$-constant family $G:Y\to \C$ over $S^1$ where $Y\to S^1$
is a priori only locally isomorphic to 
((a neighborhood of $\{0\}\times S^1$ in $\C^{n+1}\times S^1)\to S^1$) would not give more.
Then $Y\to S^1$ is globally isomorphic to 
((a neighborhood of $\{0\}\times S^1$ in $\C^{n+1}\times S^1)\to S^1$), as a $C^\iiii$-family
of neighborhoods of $0$ in $\C^{n+1}$.
\end{remarks}

\section{Moduli space of marked singularities}\label{c4}
\setcounter{equation}{0}

\noindent
Now we come to the second main subject of this paper, (strongly) marked
singularities and moduli spaces for them.

\begin{definition}\label{t4.1}
Let $f_0\in\mmm^2_{\C^{n+1},0}$ be a function germ with an isolated singularity at $0$
with Milnor number $\mu$ (short: a {\rm singularity}).
Recall $ML(f)=(Ml(f),L,M_h,I)$ from section \ref{c2}.

(a) A {\rm strongly marked singularity} is a tuple $(f,\rho)$ where $f\in\mmm^2$ is 
a singularity in the $\mu$-homotopy class of $f_0$ and 
$\rho:ML(f)\to ML(f_0)$ is an isomorphism.

(b) A {\rm marked singularity} is a tuple $(f,\pm\rho)$ with $f$ and 
$\rho$ as in (a) (writing $\pm \rho$ we mean the set $\{\rho,-\rho\}$,
neither $\rho$ nor $-\rho$ is preferred, so $(f,\pm \rho)=(f,\pm(-\rho))$).

(c) Two strongly marked singularities $(f_1,\rho_1)$ and $(f_2,\rho_2)$ 
are {\rm right equivalent} (notation: $(f_1,\rho_1)\sim_\RR(f_2,\rho_2)$)
if a coordinate change $\varphi\in\RR$ exists with
$$f_1=f_2\circ\varphi\quad\textup{and}\quad\rho_1=\rho_2\circ\varphi_{hom}.$$

(d) Two marked singularities $(f_1,\pm\rho_1)$ and $(f_2,\pm\rho_2)$ 
are {\rm right equivalent} (notation: $(f_1,\pm\rho_1)\sim_\RR(f_2,\pm\rho_2)$)
if a coordinate change $\varphi\in\RR$ exists with
$$f_1=f_2\circ\varphi\quad\textup{and}\quad(\rho_1=\rho_2\circ\varphi_{hom}
\quad\textup{ or }\quad \rho_1=-\rho_2\circ\varphi_{hom}).$$
\end{definition}

\begin{remarks}\label{t4.2}
(a) The notions {\it strongly marked} and {\it marked} are closely related,
but the first looks more natural than the second.
We use also the second notion, for two reasons:

(i) $(f,\rho)$ and $(f,-\rho)$ have the same value under the period map
to $D_{BL}$ considered in section \ref{c5}.

(ii) $(f,\rho)\sim_\RR(f,-\rho)$ if $\mult f=2$ and $(f,\rho)\not\sim_\RR(f,-\rho)$
if $\mult f\geq 3$, by theorem \ref{t3.3} (e) and (g). This implies that the moduli
space for strongly marked singularities in theorem \ref{t4.3} is not 
Hausdorff if a $\mu$-homotopy class contains singularities with multiplicity $\geq 3$
and singularities with multiplicity $2$. The moduli space for marked singularities
is not affected by this.

(b) Because of (ii), we will sometimes make one of the following two assumptions.
\begin{eqnarray}\label{4.1}
&\textup{Assumption }\eqref{4.1}:&
\textup{ Any singularity in the }\mu\textup{-homotopy class}\\ \nonumber
&&\textup{ of }f_0\textup{ has multiplicity }\geq 3.\\   \label{4.2}
&\textup{Assumption }\eqref{4.2}:&  
\textup{ Any singularity in the }\mu\textup{-homotopy class}\\ \nonumber
&&\textup{ of }f_0\textup{ has multiplicity }2.
\end{eqnarray}
For $n\neq 2$ the topological type of a singularity is constant within a 
$\mu$-homotopy class \cite[theorem (2.1)]{LR}.
Then one of the two assumptions would follow from Zariski's multiplicity conjecture.
But Zariski's multiplicity conjecture is proved essentially only for curve singularities
and quasihomogeneous singularities.
For curve singularities and quasihomogeneous singularities \eqref{4.1} or \eqref{4.2}
holds.
\end{remarks}

Theorem \ref{t4.3} and theorem \ref{t4.4} (and theorem \ref{t3.3}) 
are the main results of the paper.
Theorem \ref{t4.3} is related to \cite[theorem 13.15]{He6}. It will be proved in 
section \ref{c7}.

\begin{theorem}\label{t4.3}
Let $f_0\in\mmm_{\C^{n+1},0}^2$ be a singularity with Milnor number $\mu$ 
and $j_kf_0=f_0$ for some $k\geq \mu+1$. Fix this $k$. Define the sets
\begin{eqnarray*}
M_\mu^{smar}(f_0)&:=&\{\textup{strongly marked }(f,\rho)\, |\\ 
&& f\textup{ in the }\mu\textup{-homotopy class of }f_0\}/\sim_\RR,\\
M_\mu^{mar}(f_0)&:=&\{\textup{marked }(f,\pm\rho)\, |\\ 
&& f\textup{ in the }\mu\textup{-homotopy class of }f_0\}/\sim_\RR.
\end{eqnarray*}
(a) Recall the set $C(k,f_0)\subset\mmm^2/\mmm^{k+1}$ discussed before
lemma \ref{t3.5}. The sets
\begin{eqnarray*}
C^{smar}(k,f_0)&:=&\{(f,\rho)\, |\, f\in C(k,f_0),\\
&& \rho:ML(f)\to ML(f_0)
\textup{ an isomorphism}\},\\
C^{mar}(k,f_0)&:=&\{(f,\pm\rho)\, |\, f\in C(k,f_0),\\ 
&& \rho:ML(f)\to ML(f_0)
\textup{ an isomorphism}\}
\end{eqnarray*}
are reduced complex spaces and locally isomorphic to $C(k,f_0)$.
As sets $M_\mu^{smar}=C^{smar}(k,f_0)/j_k\RR$, $M_\mu^{mar}=C^{mar}(k,f_0)/j_k\RR,$
here $j_k\RR$ acts by 
$j_k\varphi:(f,\rho)\mapsto (f\circ j_k\varphi^{-1}, \rho\circ \varphi_{hom}^{-1})$ 
on $C^{smar}(k,f_0)$ and similarly on $C^{mar}(k,f_0)$.

(b) $C^{mar}(k,f_0)/j_k\RR$ is an analytic geometric quotient.
The induced reduced complex structure on $M_\mu^{mar}$ is independent of $k$.
The germ {}\\{} $(M_\mu^{mar},[(f,\pm \rho)])$ is isomorphic to the $\mu$-constant
stratum in a semiuniversal unfolding of $f$
(see section \ref{c7} for the $\mu$-constant stratum).
The canonical complex structure on $\mu$-constant strata from
\cite[theorem 12.4]{He6} induces a canonical complex structure on $M_\mu^{mar}$.
If not said otherwise, $M_\mu^{mar}$ will be considered with the canonical 
complex structure.

(c) For any $\psi\in G_\Z(f_0)=:G_\Z$, the map
$$\psi_{mar}:M_\mu^{mar}\to M_\mu^{mar},\quad
[(f,\pm\rho)]\to [(f,\pm\psi\circ\rho)]$$
is an automorphism of $M_\mu^{mar}$.  The action
$$G_\Z\times M_\mu^{mar},\quad 
(\psi,[(f,\pm\rho)]\mapsto \psi_{mar}([(f,\pm\rho)])$$
is a group action from the left.

(d) The action of $G_\Z$ on $M_\mu^{mar}$ is properly discontinuous.
The quotient $M_\mu^{mar}/G_\Z$ is the moduli space $M_\mu$ from
\cite[theorem 13.5]{He6} for right equivalence classes in the 
$\mu$-homotopy class of $f_0$, with its canonical complex structure.
Especially $[(f_1,\pm\rho_1)]$ and $[(f_2,\pm\rho_2)]$ are in one
$G_\Z$-orbit if and only if $f_1$ and $f_2$ are right equivalent.

(e) If assumption \eqref{4.1} or \eqref{4.2} holds then (b), (c) and (d)
are also true for $C^{smar}(k,f_0)$, $M_\mu^{smar}$ and $\psi_{smar}$
with $\psi_{smar}([(f,\rho)]):=[(f,\psi\circ\rho)]$.
If neither \eqref{4.1} nor \eqref{4.2} holds then the quotient topology
on $C^{smar}(k,f_0)/j_k\RR$ is not Hausdorff.
\end{theorem}

\begin{theorem}\label{t4.4}
Consider the same data as in theorem \ref{t4.3}.

(a) Let $(M_\mu^{mar})^0$ be the topological component of $M_\mu^{mar}$
(with its reduced complex structure) which contains $[(f_0,\pm\id)]$. Then
$$G^{mar}(f_0)=\{\psi\in G_\Z\, |\, \psi\textup{ maps }(M_\mu^{mar})^0
\textup{ to itself}\},$$
and the map
\begin{eqnarray*}
G_\Z/G^{mar}(f_0)&\to& \{\textup{topological components of }M_\mu^{mar}\}\\
\psi\cdot G^{mar}(f_0)&\mapsto& (\textup{the component }\psi_{mar}((M_\mu^{mar})^0)
\end{eqnarray*}
is a bijection.

(b) If assumption \eqref{4.1} or \eqref{4.2} holds then (a) is also true
for $M_\mu^{smar}$ and $G^{smar}(f_0)$.

(c) For any $[(f,\pm\rho)]\in M_\mu^{mar}$
$$\Stab_{G_\Z}([(f,\pm\rho)]) = \rho\circ G^{mar}_\RR(f)\circ \rho^{-1}$$
(this does not use theorem \ref{t4.3} (b)-(d)).

(d) For any $[(f,\rho)]\in M_\mu^{smar}$
\begin{eqnarray*}
\Stab_{G_\Z}([(f,\rho)]) &=& \rho\circ G^{smar}_\RR(f)\circ \rho^{-1}\\
&=& \rho\circ
\{\varphi_{hom}\, |\, \varphi\in\RR\textup{ with }f=f\circ\varphi\}\circ \rho^{-1}
\end{eqnarray*}
(this does not require assumption \eqref{4.1} or \eqref{4.2}, and it does not 
use theorem \ref{t4.3} (e)).

(e) $-\id\in G_\Z$ acts trivially on $M_\mu^{mar}(f_0)$. 
Suppose that assumption \eqref{4.2} holds and that $f_0=g_0(x_0,...,x_{n-1})+x_n^2$.
Then $-\id$ acts trivially on $M_\mu^{smar}(f_0)$ and 
\begin{eqnarray*}
\begin{matrix}
M_\mu^{smar}(f_0)&=&M_\mu^{mar}(f_0)&=&M_\mu^{mar}(g_0),\\
G^{mar}(f_0)&=&G^{smar}(f_0)&=&G^{mar}(g_0).\end{matrix}
\end{eqnarray*}
Suppose additionally that assumption \eqref{4.1} holds for $g_0$
(instead of $f_0$ in \eqref{4.1}). Then $\{\pm\id\}$ acts freely on $M_\mu^{smar}(g_0)$,
and the quotient map
$$M_\mu^{smar}(g_0) \stackrel{/\{\pm\id\}}{\longrightarrow}
M_\mu^{mar}(g_0),\quad [(f,\rho)]\mapsto [(f,\pm\rho)]$$
is a double covering.
\end{theorem}

\begin{proof}
(a) $C(k,f_0)$ is the base space for the holomorphic $\mu$-constant family
$(C(k,f_0),\C^{n+1}\times C(k,f_0),F)$ considered in lemma \ref{t3.5} (a).
By lemma \ref{t3.5} (c) and $f_0=j_kf_0$
$$G^{mar}(f_0) = \{\pm \psi\, |\, \psi\in G(F,f_0)\}.$$
The space $C^{mar}(k,f_0)$ contains the set $\{f_0\}\times G_\Z/\{\pm\id\}$.
The component $(C^{mar}(k,f_0))^0$ of $C^{mar}(k,f_0)$ which contains 
$(f_0,\pm\id)$ intersects $\{f_0\}\times G_\Z/\{\pm\id\}$ precisely in the set
$\{f_0\}\times G^{mar}/\{\pm\id\}$.

The group $G_\Z$ acts on $C^{mar}(k,f_0)$ by $\psi:(f,\pm\rho)\mapsto (f,\pm\psi\circ\rho)$
for $\psi\in G_\Z$.
The subgroup which maps the component $(C^{mar}(k,f_0))^0$ to itself is $G^{mar}(f_0)$.
The quotient map $C^{mar}(k,f_0)\to M_\mu^{mar}$ is $G_\Z$-equivariant.
As $j_k\RR$ is connected, the component $(M_\mu^{mar}(f_0))^0$ is the quotient
of the component $(C^{mar}(k,f_0))^0$. Therefore $G^{mar}(f_0)$ is the subgroup
which maps $(M_\mu^{mar}(f_0))^0$ to itself.
The bijective correspondence is clear.

(b) Similar to (a).

(c) Similar to (a). Instead of $C^{mar}(k,f_0)$ one considers the (smooth) analytic
subvariety
$$\{(g,\pm\sigma)\, |\, g\in C(k,f_0), g\sim_\RR f, \sigma:ML(g)\to ML(f_0)
\textup{ an isomorphism}\}.$$
The action of $j_k\RR$ on $C^{mar}(k,f_0)$ 
restricts to a transitive action on each component of this subvariety.
The components are not permuted as $j_k\RR$ is connected.
The components are mapped to different points in $M_\mu^{mar}$.

The component which contains $(f,\pm\rho)$ intersects 
$\{f\}\times (G_\Z/\{\pm\id\})\circ \rho$
precisely in $\{f\}\times \rho\circ(G^{mar}_\RR(f)/\{\pm\id\})$,
because of lemma \ref{t3.5} (c). Therefore the subgroup which maps this 
component to itself is $\rho\circ G^{mar}_\RR(f)\circ\rho^{-1}$.
As the quotient map is $G_\Z$-equivariant, this subgroup coincides with
$\Stab_{G_\Z}([(f,\pm\rho)])$.

(d) The first equality follows as in (c). The equality of the first and the third
term follows immediately
from the definition of right equivalence classes of strongly marked singularities.

(e) By assumption \eqref{4.2} any $[(f,\rho)]\in M_\mu^{smar}(f_0)$ satisfies
$\mult f=2$, so $(f,\rho)\sim_\RR (f,-\rho)$ by part (d) and theorem \ref{t3.3} (g).
Therefore $-\id$ acts trivially on $M_\mu^{smar}(f_0)$, 
and $M_\mu^{smar}(f_0)= M_\mu^{mar}(f_0)$.
Theorem \ref{t3.3} (b) says $G^{smar}(f_0)=G^{mar}(f_0)$.
With the two isomorphisms
$$Ml(g)\to Ml(g)\otimes Ml(x_n^2) =Ml(g+x_n^2)$$
which just differ by a sign, the map
$$M_\mu^{mar}(g_0)\to M_\mu^{mar}(f_0),\quad 
[(g,\pm\sigma)]\mapsto [(g+x_n^2,\pm \sigma\otimes\id)]$$
is well defined. It is surjective because of the splitting lemma 
and assumption \eqref{4.2} for $f_0$. 
It is bijective because of part (c) and theorem \ref{t3.3} (h).
Part (a) and $M_\mu^{mar}(f_0)=M_\mu^{mar}(g_0)$ show $G^{mar}(f_0)=G^{mar}(g_0)$.

By assumption \eqref{4.1} for $g_0$, $\mult g\geq 3$ 
for any $(g,\sigma)\in M_\mu^{smar}(g_0)$, so $(g,\sigma)\not\sim_\RR (g,-\sigma)$
by part (d) and theorem \ref{t3.3} (g). Therefore $\pm\id$ acts freely on $M_\mu^{smar}(g_0)$
and the map $M_\mu^{smar}(g_0)\to M_\mu^{mar}(g_0)$ is a double covering.
\end{proof}

\begin{remarks}\label{t4.5}
(i) Theorem \ref{t4.4} (a) shows that conjecture \ref{t3.2} (a) is equivalent to the 
connectedness of $M_\mu^{mar}(f_0)$. Theorem \ref{t4.4} (b) shows that 
conjecture \ref{t3.2} (b) is equivalent to $[(f,\rho)]$ and $[(f,-\rho)]$
being in different components of $M_\mu^{smar}(f_0)$ if assumption \eqref{4.1} holds.
Together the conjectures say that under assumption \eqref{4.1} $M_\mu^{smar}(f_0)$
has two components, each isomorphic to $M_\mu^{mar}(f_0)$, and that they are 
permuted by $-\id$.

(ii) Theorem \eqref{t4.4} (a) and (b) use that $M_\mu^{mar}(f_0)$ 
respectively $M_\mu^{smar}(f_0)$ is an analytic quotient. But Theorem \ref{t4.4} (c)
and (d) use only the (trivial) equality $C^{mar}(k,f_0)/j_k\RR = M_\mu^{mar}$ as sets
(respectively for smar). In part (e) only the equality $G^{mar}(f_0)=G^{mar}(g_0)$
uses that $M_\mu^{mar}(f_0)$ and $M_\mu^{mar}(g_0)$ are analytic geometric
quotients. The other statements can be understood and proved without this.
\end{remarks}

\section{Period maps and Torelli type problems}\label{c5}
\setcounter{equation}{0}

\noindent
In \cite{He1} I had defined an analytic invariant $LBL(f)$ of the right equivalence class
of a singularity $f$ and had formulated the Torelli type conjecture that $LBL(f)$
determines $f$ up to right equivalence.
I worked on it in \cite{He1}--\cite{He6}. It is reformulated in conjecture \ref{t5.4}.
Using $M_\mu^{mar}$, now a stronger conjecture for marked singularities can be proposed,
conjecture \ref{t5.3}. 

First, the invariant $LBL(f)$ will be described, but with the minimum of details
necessary to appreciate it. More detailed accounts can be found in 
\cite{He1}--\cite[ch. 10]{He6}. It builds on the Brieskorn lattice and the 
Gauss-Manin connection, which had been studied in many ways,
e.g. \cite{Ma1}\cite{Va}\cite[III]{AGV2}\cite{SchS}\cite{Sa1}\cite{Sa2}\cite{Ku}.

Fix a singularity $f\in\mmm_{\C^{n+1},0}^2$ and a Milnor fibration
$f:Y'(\varepsilon,\delta)\to T_\delta'$ for it as in section \ref{c2}.
The cohomology bundle
$$H^n_\C :=\bigcup_{\tau\in T_\delta'}H^n(f^{-1}(\tau),\C) 
\supset H^n_\Z :=\bigcup_{\tau\in T_\delta'}H^n(f^{-1}(\tau),\Z) $$
is a flat vector bundle of rank $\mu$ and contains a flat $\Z$-lattice bundle.
Denote by $H^\infty\supset H^\infty_\Z$ the spaces of global flat multivalued sections
in $H^n_\C$ respectively in $H^n_\Z$.

A holomorphic section can be written as a linear combination of a basis of 
$H^\infty_\C$ (or $H^\infty_\Z$) with multivalued holomorphic coefficients.
A section in a punctured neighborhood of $0$ is said to have moderate growth
respectively vanishing growth if these coefficients have moderate growth
respectively vanishing growth. The spaces of germs at 0 of such sections 
are denoted $V^{>-\iiii}$ and $V^{>0}$. Write also $V^{>\alpha}:=\tau^\alpha V^{>0}$
for $\alpha\in\Z$. Then $V^{>-\iiii}$ is a $\C\{\tau\}[\tau^{-1}]$-vector space
of dimension $\mu$ and $V^{>\alpha}$ is a free $\C\{\tau\}$-module of rank $\mu$
in it with $V^{>\alpha}\otimes_{\C\{\tau\}}\C\{\tau\}[\tau^{-1}]=V^{>-\iiii}$.

The Brieskorn lattice $H_0''(f)$ is a free $\C\{\tau\}$-module of rank $\mu$
whose (germs of) sections
come from differential forms as follows:
For $\omega\in\Omega^{n+1}_{Y(\varepsilon,\delta)}$ the section 
$s[\omega]$ on $T_\delta'$ with value 
$$s[\omega](\tau) := \left[{\frac{\omega}{\ddd f}}_{|f^{-1}(\tau)}\right]\in H^n(f^{-1}(\tau),\C)$$
is holomorphic, its germ $s[\omega]_0$ at $0$ turns out to be in $V^{>-1}$
\cite{Ma1}. $H_0''(f)$ is generated by such germs. 
Therefore $V^{>-1}\supset H_0''(f)$. Also $H_0''(f)\supset V^{>n-1}$ holds.

The Brieskorn lattice is a rich invariant. 
It induces a (sum of two) polarized mixed Hodge structure(s) on $H^\iiii\supset H^\iiii_\Z$
(the mixed Hodge structure: \cite{Va}\cite{SchS}\cite{Sa1}, 
its polarization: \cite{He4}, but see \cite[remark 10.25]{He6} for a sign mistake in
\cite{He4}).

Any fiber $H^n_{\Z,\tau}$ for $\tau>0$ is canonically isomorphic to the dual 
of the Milnor lattice $Ml(f)$. Therefore the Milnor lattice $Ml(f)$ and its monodromy
$M_h$ determine uniquely $H^n_\Z$, $H^\iiii_\Z$, $V^{>-\iiii}$, $V^{>\alpha}$.
Any automorphism of $(Ml(f),M_h)$ induces automorphisms of
$H^n_\Z$, $H^\iiii_\Z$, $V^{>-\iiii}$, $V^{>\alpha}$, which will be denoted by the same letter
as the original automorphism (here $\psi\in\Aut(Ml(f),M_h)$ induces
$\psi(\gamma):=\gamma\circ \psi^{-1}$ for $\gamma\in H^n_{\Z,\tau}$).
Like $Ml(f)$, the germ at $0$ of the bundle $H^n_\Z$ and the spaces 
$H^\iiii_\Z$, $V^{>-\iiii}$, $V^{>\alpha}$, $H_0''(f)$ are independent of the choice
of the Milnor fibration. The tuple 
$$(ML(f),\textup{ the germ at 0 of }H^n_\Z,V^{>-\iiii},H_0''(f))$$
will be abbreviated $(ML(f), V^{>-\iiii},H_0''(f))$. It is a datum of the germ
$f\in\mmm_{\C^{n+1},0}^2$ (and not of some special representative).
The invariant $LBL(f)$ is the isomorphism class of 
$(ML(f),V^{>-\iiii},H_0''(f))$. It is a datum of the right equivalence class of $f$.

The group 
$$\Stab_{G_\Z(f)} (H_0''(f)) := \Aut(ML(f),V^{>-\iiii},H_0''(f))\subset G_\Z(f)$$
is finite because of the polarized mixed Hodge structure on $H^\iiii\supset H^\iiii_\Z$
\cite{He4}.
If $\varphi\in\RR$ with $f=f\circ\varphi$ then $\varphi_{hom}\in \Stab_{G_\Z(f)}(H_0''(f))$,
because $H_0''(f)$ is defined by the geometry and is independent of the choice
of coordinates. Therefore (and because of theorem \ref{t3.3} (e))
$$G^{mar}_\RR(f)\subset \Stab_{G_\Z(f)}(H_0''(f)),$$
and $G^{mar}_\RR(f)$ is also finite. Conjecture \ref{t5.1} for $G^{mar}_\RR(f)$
is similar to conjecture \ref{t3.2} (a) for $G^{mar}(f)$.

\begin{conjecture}\label{t5.1}
\cite[conjecture 13.12]{He6} 
$$G^{mar}_\RR(f) = \Stab_{G_\Z(f)}(H_0''(f)).$$
\end{conjecture}

See theorem \ref{t5.6} for some cases in which it holds.

In a $\mu$-constant family $(X,Y,F)$ as in definition \ref{t2.1}, locally the Milnor lattices
with Seifert forms $ML(F_t)$ are canonically isomorphic (theorem \ref{t2.2}).
Therefore also the germs of bundles $H^n_\Z(F_t)$ and the spaces 
$H^\iiii_\Z(F_t)$, $V^{>-\iiii}(F_t)$, $V^{>\alpha}(F_t)$ are canonically isomorphic.
But the Brieskorn lattices vary holomorphically.

Now fix one singularity $f_0\in \mmm_{\C^{n+1},0}^2$. In \cite{He4} a classifying space
$D_{BL}(f_0)$ for $\C\{\tau\}$-lattices in $V^{>-\iiii}(f_0)$ which have many properties
of Brieskorn lattices was constructed, a {\it classifying space for Brieskorn lattices}.
It is a complex manifold. 

Let $(f,\pm\rho)$ be a marked singularity, so $f$ is in the $\mu$-homotopy class of $f_0$
and $\rho:ML(f)\to ML(f_0)$ is an isomorphism. Then $\rho$ induces an isomorphism
$\rho:V^{>-\iiii}(f)\to V^{>-\iiii}(f_0)$, and $\rho(H_0''(f))$ is a point in $D_{BL}(f_0)$.
One obtains a period map
$$BL:M_\mu^{mar}(f_0)\to D_{BL}(f_0),\quad [(f,\pm\rho)]\mapsto \rho(H_0''(f)).$$
Locally, $M_\mu^{mar}(f_0)$ is isomorphic to a $\mu$-constant stratum (theorem \ref{t4.3} (b)).
Locally, this period map had been studied in \cite{Sa1}\cite{Sa2}\cite{He1}--\cite{He6},
and it is holomorphic, so $BL$ is holomorphic.

\begin{theorem}\label{t5.2}
\cite[theorem 12.8]{He6} $BL$ is an immersion, here the reduced complex structure
on $M_\mu^{mar}(f_0)$ is considered.
\end{theorem}

This improves a slightly weaker result (finite-to-one) in \cite{Sa2}.
It is an infinitesimal Torelli type result. 
I have some evidence that $BL$ is also an immersion with the canonical complex
structure on $M_\mu^{mar}(f_0)$. But if true, this will be subject of another paper.

The following is a global Torelli type conjecture for marked singularities.

\begin{conjecture}\label{t5.3}
(New) $BL$ is injective, here the canonical complex structure
on $M_\mu^{mar}(f_0)$ is considered.
\end{conjecture}

The group $G_\Z(f_0)$ acts on $M_\mu^{mar}(f_0)$ and on $D_{BL}(f_0)$ properly 
discontinuously, and by its definition the period map $BL$ is $G_\Z(f_0)$-equivariant.
The quotient $M_\mu^{mar}(f_0)/G_\Z(f_0)$ is the moduli space $M_\mu(f_0)$
of right equivalence classes of singularities in the $\mu$-homotopy class of $f_0$
(theorem \ref{t4.3} (d)), the quotient $D_{BL}(f_0)/G_\Z(f_0)$ is a moduli space
for the invariants $LBL(f)$. One obtains a period map
$$LBL:M_\mu(f_0)\to D_{BL}(f_0)/G_\Z(f_0),\quad [f]\mapsto LBL(f).$$
The global Torelli type conjecture for right equivalence classes of singularities
from \cite{He1} says that $LBL$ is injective where the reduced complex
structure on $M_\mu(f_0)$ is considered. It can be strengthened as follows.

\begin{conjecture}\label{t5.4}
\cite[conjecture 12.7]{He6} $LBL$ is injective, here the canonical complex
structure on $M_\mu(f_0)$ is considered.
\end{conjecture}

\begin{lemma}\label{t5.5}
Conjecture \ref{t5.3} $\iff$ conjecture \ref{t5.1} (for all singularities $f$ in the 
$\mu$-homotopy class of $f_0$) and conjecture \ref{t5.4}.
\end{lemma}

\begin{proof}
Conjecture \ref{t5.4} says that $\rho(H_0''(f)) \neq \www\rho(H_0''(\www f))$ for 
all possible markings $\pm\rho$ and $\pm\www\rho$ of $f$ and $\www f$
if and only if $f\not\sim_\RR \www f$. In the case $f=\www f$, 
$\rho(H_0''(f))=\www\rho(H_0''(f))$ is equivalent to 
$\www\rho^{-1}\circ\rho\in\Stab_{G_\Z(f)}(H_0''(f))$. By conjecture \ref{t5.1}
this is equivalent to $(f,\pm\rho)\sim_\RR (f,\pm\www\rho)$.

This shows the equivalence in the case of the reduced complex structures.
For the canonical complex structures, one observes that the one on $M_\mu(f_0)$
is induced by the one on $M_\mu^{mar}(f_0)$ by the quotient map
$M_\mu^{mar}(f_0)\to M_\mu(f_0)=M_\mu^{mar}(f_0)/G_\Z(f_0)$ (theorem \ref{t4.3} (d)).
\end{proof}

\begin{theorem}\label{t5.6}
(a) Conjecture \ref{t5.4} is true for all singularities with modality $\leq 2$
possibly with the exception of the subseries $Z_{1,14k}$, $S_{1,10k}$, 
$S_{1,10k}^\sharp$ ($k\geq 1$) \cite{He1}\cite{He2}.
It is true for the $\mu$-homotopy classes of the Brieskorn-Pham singularities
$\sum_{i=0}^n x_i^{a_i}$ with pairwise coprime exponents and for the $\mu$-homotopy
class of the singularity $x_0^3+x_1^3+x_2^3+x_3^3$ \cite{He3}.

(b) \cite{He5} $\Stab_{G_\Z}(H_0''(f))=\{\pm\id\}$ for generic semiquasihomogeneous
singularities with $\frac{n+1}{2}-\sum_{i=0}^n w_i \geq 4$.

(c) \cite{He5} The period map $LBL$ is generically injective for the semiquasihomogeneous
singularities with $\frac{n+1}{2}-\sum_{i=0}^n w_i \geq 4$.

(d) Conjecture \ref{t5.3} (and by lemma \ref{t5.5} also the conjectures \ref{t5.1}
and \ref{t5.4}, but the last one is known since \cite{He1}) is true for all 
singularities listed in theorem \ref{t3.3} (a).
\end{theorem}

The new part of this theorem is part (d). For the simple, 22 of the 28 exceptional and the
Brieskorn-Pham singularities it will be proved in section \ref{c8}.
The remaining 6 exceptional, 
the simple-elliptic and the hyperbolic singularities will be treated in another paper.
In all these cases one can build on the study of the period map 
$LBL:M_\mu\to D_{BL}/G_\Z$ in \cite{He1}.
The crucial new point is to determine $M_\mu^{mar}$. 
And central for this is to prove conjecture \ref{t3.2} (a).
Then $M_\mu^{mar}$ has only one component by theorem \ref{t4.4} (a).

One may expect that $\Stab_{G_\Z}(H_0''(f))=\{\pm\id\}$ for generic singularities
in one $\mu$-homotopy class, if that class is not too small. Part (b) shows this
for semiquasihomogeneous singularities with $\frac{n+1}{2}-\sum_{i=0}^n w_i \geq 4$. 
But special members may have a large finite isotropy group of their Brieskorn lattice.

\section{Symmetries of singularities}\label{c6}
\setcounter{equation}{0}

\noindent
Here we will review some results on symmetries of singularities from \cite[13.1 and 13.2]{He6}
and simplify the proofs. The results will imply theorem \ref{t3.3} (f)+(g)+(h).
They will also be used in the proof of theorem \ref{t4.3} in section \ref{c7}.
They build on work of Slodowy \cite{Sl} and Wall \cite{Wa1}\cite{Wa2}.

An {\it unfolding} of a singularity $f\in\mmm_{\C^{n+1},0}^2$(a germ with an isolated singularity
at $0$) is a holomorphic function germ $F:(\C^{n+1}\times M,0)\to (\C^{n+1},0)$
with $F_{|(\C^{n+1}\times\{0\},0)}=f$ and $(M,0)$ a germ of a manifold.
It is {\it versal} if it induces any unfolding $G:(\C^{n+1}\times N,0)\to (\C,0)$,
that means, a morphism $\varphi:(N,0)\to (M,0)$ and a morphism 
$\Phi:(\C^{n+1}\times N,0)\to (\C^{n+1}\times M,0)$ exist such that
$$\pr_M\circ\Phi=\varphi\circ\pr_N\quad\textup{and}\quad G=F\circ \Phi\quad\textup{and}
\quad\Phi_{|(\C^{n+1}\times\{0\},0)}=\id.$$
It is {\it semiuniversal} if it is versal und if $\dim (M,0)$ is minimal.

Semiuniversal unfoldings exist (see e.g. \cite{Ma2}\cite{AGV1}).
Denote by $J_F:=(\ppp F/\ppp x_i)\subset \OO_{\C^{n+1}\times M,0}$ the 
{\it Jacobi ideal} of $F$, and by $J_f$ the one of $f$.
An unfolding is semiuniversal if and only if the map
$$\aaa_C:\TT_{M,0}\to \OO_{\C^{n+1}\times M,0}/J_F,\quad \frac{\ppp}{\ppp t_i}\mapsto
\left[\frac{\ppp F}{\ppp t_i}\right]$$
is an isomorphism. Equivalent is that the map
$$\aaa_0:T_0M\to \OO_{\C^{n+1},0}/J_f$$
is an isomorphism.
Then $\aaa_C$ induces a multiplication $\circ$ on $\TT_{M,0}$, a unit vector field
$e:=\aaa_C^{-1}(1)$ and an Euler field $E:=\aaa_C^{-1}([F])$.
Then $((M,0),\circ, e,E)$ is the germ of an {\it F-manifold with Euler field},
see \cite{HM}\cite{He6} for its definition.

One can choose good representatives $F$ and $M$. Then for all $t\in M$ the sum of the
Jacobi algebras of the critical points of $F_t$ is isomorphic via $\aaa_C$ to
$T_tM$ as an algebra. For generic $t\in M$ $F_t$ has only $A_1$-singularities,
so the multiplication on $TM$ is generically semisimple. Such an F-manifold is
called {\it massive}.
The group $\Aut_M:=\Aut((M,0),\circ,e,E)$ of automorphisms of a germ of a massive
F-manifold with Euler field is finite \cite[theorem 4.14]{He6}.

Denote by
$$\RR^f:=\{\varphi\in\RR\, |\, f=f\circ \varphi\}$$
the group of {\it  symmetries} of $f$.
 
Consider a semiuniversal unfolding $F$ of $f$ with base space $(M,0)$ and
a symmetry $\varphi$ of $f$. 
Then $F\circ\varphi^{-1}$ is  a semiuniversal unfolding of $f$ with the same 
base space $(M,0)$. It is induced by $F$ via a pair $(\Phi,\varphi_M)$ of isomorphisms
with
$$F\circ \varphi^{-1}=F\circ\Phi,\quad \Phi_{|(\C^{n+1}\times \{0\},0)}=\id,
\quad \varphi_M\in\Aut_M.$$
Then $\www\Phi:=\Phi\circ\varphi$ satisfies
\begin{eqnarray}\label{6.1}
\pr_M\circ \www\Phi=\varphi_M\circ\pr_M,\quad F=F\circ\www \Phi,\quad
\www\Phi_{|(\C^{n+1}\times \{0\},0)}=\varphi.
\end{eqnarray}
Here $\Phi$ is not at all unique, but $\varphi_M$ is unique because $\Aut_M$ is finite
and $\aaa_0\circ (\ddd\varphi_M)_{|0}\circ \aaa_0^{-1}$ is the automorphism
of $\OO_{\C^{n+1},0}/J_f$ which is induced by $\varphi$.
One obtains a group homomorphism
$$()_M:\RR^f\to \Aut_M,\quad \varphi\mapsto \varphi_M.$$
The group $\RR^f$ is possibly $\infty$-dimensional, 
but the group $j_k\RR^f$ of $k$-jets in $\RR^f$
is an algebraic group for any $k$. Let
$$R_f:= j_1\RR^f/(j_1\RR^f)^0$$
be the finite group of components of $j_1\RR^f$. It is easy to see that
$R_f=j_k\RR^f/(j_k\RR^f)^0$ for any $k\geq 1$ \cite[lemma 13.8]{He6}.
The following theorem is contained in \cite[theorem 13.9]{He6},
except for part (d). Some parts of the proof below are simpler than in \cite{He6}.

\begin{theorem}\label{t6.1}
Fix a singularity $f\in\mmm_{\C^{n+1},0}^2$ and a semiuniversal unfolding $F$ with 
base space $(M,0)$.

(a) The homomorphism $()_M:\RR^f\to \Aut_M$ factors through $R_f$ to a 
homomorphism $()_M:R_f\to \Aut_M$.

(b) If $\mult f\geq 3$ then $()_M:R_f\to \Aut_M$ is an isomorphism.
And then $j_1\RR^f=R_f.$

(c) If $\mult f=2$ then $()_M:R_f\to \Aut_M$ is surjective with kernel of order $2$. 
If $f=g(x_0,...,x_{n-1})+x_n^2$ then the kernel is generated by the class of the
symmetry $(x\mapsto (x_0,...,x_{n-1},-x_n))$.

(d) If $\mult f=2$ denote by $\varphi^{(1)}$ and $\varphi^{(2)}$ the (linear) actions
of $\varphi\in\RR^f$ on $\frac{\mmm}{\mmm^2+J_f}$ and on $\frac{\mmm^2+J_f}{\mmm^2}$.
Then $\det \varphi^{(2)}\in\{\pm 1\}$. The homomorphism
$$\RR^f\to \Aut\left(\frac{\mmm}{\mmm^2+J_f}\right)\times\{\pm 1\},\quad
\varphi\mapsto (\varphi^{(1)},\det\varphi^{(2)})$$
factors through $R_f$ to an injective homomorphism 
$$R_f\to \Aut\left(\frac{\mmm}{\mmm^2+J_f}\right)\times\{\pm 1\}.$$

(e) If $f=g(x_0,..,x_m)+x_{m+1}^2+...+x_n^2$ and $\mult g\geq 3$ then
$R_f=R_g\times (\ker ()_M)\cong R_g\times S_2$. If $\mult g=2$ then
$R_f=R_g$.

(f) The homomorphism $()_{hom}:\RR^f\to G_\Z(f)$ factors through $R_f$ to a 
homomorphism $()_{hom}:R_f\to G_\Z(f)$. 
The image is 
$$(R_f)_{hom}=G^{smar}_\RR(f)\subset \Stab_{G_\Z}(H_0''(f))\subset G_\Z(f).$$

(g) The homomorphism $()_{hom}:R_f\to G^{smar}_\RR(f)$ is an isomorphism.

(h) The homomorphism $()_M\circ ()_{hom}^{-1}:G^{smar}_\RR(f)\to\Aut_M$
is an isomorphism if $\mult f\geq 3$. It is 2--1 with kernel $\{\pm \id\}$
if $\mult f=2$. In any case it extends to a 2--1 morphism
$()_{hom\to M}:G^{mar}_\RR(f)\to\Aut_M$ with kernel $\{\pm\id\}$.

(i) Theorem \ref{t3.3} (f)+(g)+(h) is true.
\end{theorem}

\begin{proof}
(a) The action of $\varphi\in \RR^f$ on $\OO_{\C^{n+1},0}/J_f$ depends only on
a sufficiently high $k$-jet of $\varphi$, and it depends continuously on it.
Because of $\OO/J_f\cong T_0M$ and because $\Aut_M$ is finite, $k$-jets of 
symmetries in one component of $j_k\RR^f$ induce the same element of $\Aut_M$.

(b) Surjectivity: The F-manifold $(M,\circ,e,E)$ determines a Lagrange variety
in $T^*M$, and this determines up to isomorphism a semiuniversal unfolding,
see \cite[theorem 5.6]{He6} for details and \cite[19.3]{AGV1} for the relation
between Lagrange maps and unfoldings. Therefore any $\varphi_M\in\Aut_M$ lifts
to an automorphism $(\www\Phi,\varphi_M)$ of the unfolding $F$, with 
$\www\Phi$ as in \eqref{6.1}.

Injectivity: $\mult f\geq 3$ implies $J_f\subset\mmm^2$ and
surjectivity of the map $\mmm/J_f\to \mmm/\mmm^2$.
The action of $j_1\RR^f$ on $\mmm/\mmm^2$ is faithful.
Therefore $j_1\RR^f=R_f$ and $()_M:R_f\to \Aut_M$ is injective.

(c)+(d)+(e) Surjectivity in (c): as in (b).

It is sufficient to consider the case $f=g(x_0,...,x_m)+x_{m+1}^2+...+x_n^2$
with $\mult g\geq 3$, $m<n$. Then
$$\frac{\mmm}{\mmm^2+J_f}=\frac{\mmm}{\mmm^2+(x_{m+1},...,x_n)}\quad
\textup{and}\quad \frac{\mmm^2+J_f}{\mmm^2}=\frac{\mmm^2+(x_{m+1},...,x_n)}{\mmm^2}.$$
The kernel of the natural homomorphism
$$j_1\RR^f\to \Aut\left(\frac{\mmm}{\mmm^2+J_f}\right)
\times \Aut\left(\frac{\mmm^2+J_f}{\mmm^2}\right)$$
is unipotent, so connected. The image is $R_g\times O(n-m)$.
The second factor is due to $j_2f=x_{m+1}^2+...+x_n^2$. For the first factor 
observe the following: 
$(M,0)$ is also the base space of a semiuniversal unfolding of $g$,
$()_M:R_g\to\Aut_M$ is an isomorphism, $()_M:R_f\to \Aut_M$ is surjective,
so $R_g$ and $R_f$ induce the same automorphisms of $T_0M$.
Also
$$T_0M \cong \frac{\OO_{\C^{m+1},0}}{J_g}\supset \frac{\mmm_{\C^{m+1},0}}{J_g} 
\twoheadrightarrow
\frac{\mmm_{\C^{m+1},0}}{\mmm^2_{\C^{m+1},0}}
\cong \frac{\mmm}{(\mmm^2+J_f)},$$
and $j_1\RR^g=R_g$ acts faithfully on $\mmm_{\C^{m+1},0}/\mmm^2_{\C^{m+1},0}$
Therefore the image above is $R_g\times O(n-m)$.

As the kernel of $j_1\RR^f\to R_g\times O(n-m)$ is connected
\begin{eqnarray*}
R_f &=&j_1\RR^f/(j_1\RR^f)^0 = R_g\times O(n-m)/(R_g\times O(n-m))^0\\
&=&R_g\times \{\id,
(\textup{the class of }(x\mapsto (x_0,...,x_{n-1},-x_n)\textup{ in }R_f)\}.
\end{eqnarray*}
The rest is clear now, too.

(f) $G_\Z(f)$ is a discrete group. $(R_f)_{hom}=G^{smar}(f)$ is theorem \ref{t3.3} (e).

(g) It rests to show that $()_{hom}:R_f\to G_\Z(f)$ is injective.
Suppose $\varphi_{hom}=\id$ for some $\varphi\in\RR^f$.
Then $\varphi_{hom}$ acts trivially on $H_0''(f)$. The space
$$\Omega_f:=\Omega_{\C^{n+1},0}^{n+1}/\ddd f\land \Omega_{\C^{n+1},0}^n$$
is a quotient of $H_0''(f)$ (due to Brieskorn, see e.g. \cite{Ma1}\cite{AGV2}\cite{He6}).
It is a free $\OO/J_f$-module of rank $1$, generated by the class $[\omega_0]$
of the volume form $\omega_0:=\ddd x_0\land...\land\ddd x_n$.
The action of $\varphi$ on $\Omega_f$ is trivial, because the action on $H_0''(f)$ is trivial.
Therefore $[\varphi^*\omega_0]=[\omega_0]$ and the action of $\varphi$ on $\OO/J_f$ is trivial.
Therefore $\det\left(\frac{\ppp \varphi_j}{\ppp x_i}\right)(0)=1$ and $\varphi_M=\id$.
Because of (b) and (c) the class of $\varphi$ in $R_f$ is $\id$.

(h) For $f=g(x_0,...,x_{n-1})+x_n^2$ the proof of theorem \ref{t3.3} (b) shows  
$$(x\mapsto (x_0,...,x_{n-1},-x_n))_{hom}=-\id.$$
Thus if $\mult f=2$ then $-\id\in G^{smar}_\RR(f)$ and 
$(\{\pm\id\})^{-1}_{hom}=\ker ()_M$. 

If $\mult f\geq 3$ then $()_M\circ ()_{hom}^{-1}:G^{smar}_\RR(f)\to \Aut_M$ is
an isomorphism, and the extension to $G^{smar}_\RR(f+x_{n+1}^2)$ is 2-1 with 
kernel $\{\pm\id\}$, thus $-\id\notin G^{smar}_\RR(f)$.

(i) Compare (e)--(h).
\end{proof}

In the case of a quasihomogeneous singularity the group $R_f$ has a canonical lift
to $\RR^f$. It will be useful for the calculation of $R_f$.

\begin{theorem}\label{t6.2}
\cite[theorem 13.11]{He6}
Let $f\in \C[x_0,...,x_n]$ be a quasihomogeneous polynomial with an isolated singularity
at $0$ and weights $w_0,...,w_n\in\Q\cap(0,\frac{1}{2}]$ and weighted degree $1$.
Suppose that $w_0\leq ...\leq w_{n-1}<\frac{1}{2}$ (then $f\in \mmm^3$
if and only if $w_n<\frac{1}{2})$. Let $G_w$ be the algebraic group of 
quasihomogeneous coordinate changes, that means, those which respect $\C[x_0,...,x_n]$
and the grading by the weights $w_0,...,w_n$ on it. Then
$$R_f \cong \Stab_{G_w}(f).$$
\end{theorem}

\section{Proof of theorem \ref{t4.3}}\label{c7}
\setcounter{equation}{0}

\noindent
The proof of theorem \ref{t4.3} will be similar to the proof of theorem 13.15 in \cite{He6}.
Like that proof it will use results from \cite[13.3]{He6}, they are reformulated in 
theorem \ref{t7.2}. But it will use also joint consequences of these results and
theorem \ref{t6.1}, they are formulated in corollary \ref{t7.3}. The proof of theorem
\ref{t4.3} comes after it. The results in \cite[13.3]{He6} concern the $\mu$-constant
stratum.

For a moment, fix a singularity $f\in\mmm_{\C^{n+1},0}^2$ and choose a good representative
$F:U\times M\to \C$ with $U\subset \C^{n+1}$ of a semiuniversal unfolding,
with base space $M$. The $\mu$-constant stratum $S_\mu\subset M$ is
\begin{eqnarray*}
S_\mu &=& \{t\in M\, |\,  \textup{Crit}(F_t)=\{x\}\textup{ and }F_t(x)=0\}\\
&=& \{t\in M\, |\,  0\textup{ is the only critical value of }F_t\}\\
&=& \{t\in M\, |\, E\circ \textup{ is nilpotent on }T_tM\}.
\end{eqnarray*}
The second equality is due to Gabrielov \cite{Ga}, Lazzeri and L\^e.
The third equality follows from the definition of multiplication and Euler
field on $M$ 
(see section \ref{c6}): The eigenvalues of $E\circ$ on $T_tM$ are the critical values
of $F_t$.

The germ $(S_\mu,0)\subset (M,0)$ is a datum of the germ $(M,0)$ of an F-manifold
with Euler field, any automorphism $\psi\in \Aut_M:=\Aut((M,0),\circ,e,E)$
restricts to an automorphism of $(S_\mu,0)$.

The critical points $x$ of $F_t$ with $t\in S_\mu$ might a priori not be equal to $0$.
But by a result of Teissier \cite[6.14]{Te} there exists a holomorphic section
$\sigma:M\to U\times M$ with $\textup{Crit}(F_t)=\{\sigma(t)\}$ for $t\in S_\mu$.
Because $F(x+\sigma(t),t)$ is also a semiuniversal unfolding of $f$, we can assume
from now on that for $t\in S_\mu$ $\textup{Crit}(F_t)=\{0\}$,
Then the restriction of $F$ to $S_\mu$ is a holomorphic $\mu$-constant family
in the sense of definition \ref{t2.1}. By theorem \ref{t2.2} it comes equipped
with a flat bundle $ML(F_{|S_\mu})$ of Milnor lattices with Seifert forms $L$.

Any $\varphi\in \Aut_M$ lifts to an automorphism of the (germ of the) unfolding $F$
(see the proof of theorem \ref{t6.1} (b) and \cite[theorem 5.6]{He6}).
Therefore one may expect that $(S_\mu,0)/\sim_\RR = (S_\mu,0)/\Aut_M$.
This is true and part of much stronger results in \cite[13.3]{He6}.
They are cited in theorem \ref{t7.2}. The existence of an unfolding
with the properties in definition \ref{t7.1} is part of them.

\begin{definition}\label{t7.1}
Fix a singularity $f\in\mmm_{\C^{n+1},0}^2$. A representative $F:U\times M\to \C$
of a semiuniversal unfolding,
with $U\subset \C^{n+1}$ a neighborhood of $0$ and $M$ the base space,
is called a {\rm very good representative} if the following holds.

(i) Any $\varphi\in \Aut_M$ extends (from the germ) to an automorphism of the
F-manifold $M$. 

(ii) Any isomorphism $\psi: (M,t)\to (M,\www t)$ of germs of F-manifolds with
Euler fields and with $t,\www t\in S_\mu$ is the restriction of an element of $\Aut_M$.

(iii) Any $\varphi\in \Aut_M$ lifts to an automorphism $(\Phi,\varphi)$ of the
unfolding $F$, that means, $\Phi:Y_1\to Y_2$ is an isomorphism
of suitable open subsets $Y_1$ and $Y_2$ of $U\times M$ which contain all critical
points of all $F_t$, with $\pr_M\circ\Phi =\varphi\circ\pr_M$ and $F=F\circ\Phi$.

(iv) $S_\mu$ is contractible. Therefore $ML(F_t)$ for $t\in S_\mu$ can and will
be identified with $ML(f)$.
\end{definition}

Theorem \ref{t7.2} collects the main results of \cite[13.3]{He6}. 
We will not review the proofs here. They use the construction of Frobenius
manifolds on the base spaces of semiuniversal unfoldings and an interplay 
of this with the polarized mixed Hodge structures on the spaces $H^\iiii$
from section \ref{c5}.

\begin{theorem}\label{t7.2}
Fix a singularity $f\in\mmm_{\C^{n+1},0}^2$. 

(a) \cite[theorem 13.18]{He6} A very good representative of a semiuniversal unfolding
exists. For such a representative $S_\mu/\sim_\RR = S_\mu/\Aut_M$.

(b) \cite[theorem 13.17]{He6} If $\www f\in\mmm_{\C^{n+1},0}^2$ is a singularity 
in the $\mu$-homotopy class of $f$, but not right equivalent to $f$, then 
very good representatives $F$ and $\www F$ of semiuniversal unfoldings
of $f$ and $\www f$ exist with $\mu$-constant strata $S_\mu\subset M$ and
$\www S_\mu\subset \www M$ such that $F_t\not\sim_\RR \www F_{\www t}$
for any $t\in S_\mu$ and any $\www t\in \www S_\mu$.

(c) \cite[theorem 13.15]{He6}
Recall the space $C(k,f)\subset\mmm^2/\mmm^{k+1}$ from section \ref{c3}.
For $k\geq \mu+1$, the space $C(k,f)/j_k\RR$ is an analytic geometric quotient.
It is a moduli space for the right equivalence classes in the $\mu$-homotopy class of $f$.
Locally at $[f]$ it is isomorphic to $S_\mu/\Aut_M$ where $S_\mu$ is the
$\mu$-constant stratum of a very good representative of a semiuniversal unfolding of $f$.
A priori it carries the induced reduced complex structure. But it comes also equipped
with a canonical complex structure induced by that on $\mu$-constant strata
in \cite[theorem 12.4]{He6}.
\end{theorem}

\begin{corollary}\label{t7.3}
Fix two singularities $f_0$ and $f\in \mmm_{\C^{n+1},0}^2$ in the same $\mu$-homotopy class.
Fix a very good representative $F:U\times M\to \C$ of a semiuniversal unfolding
of $f$. Fix two isomorphisms $\rho$ and $\www\rho:ML(f)\to ML(f_0)$.
Suppose that $(F_t,\pm\rho)\sim_\RR(F_{\www t},\pm\www\rho)$ for some $t,\www t\in S_\mu$
[respectively that $(F_t,\rho)\sim_\RR(F_{\www t},\www\rho)$ and that
assumption \eqref{4.1} or \eqref{4.2} holds].

Then $\www\rho^{-1}\circ \rho\in G^{mar}_\RR(f)$ 
[respectively $\www\rho^{-1}\circ \rho\in G^{smar}_\RR(f)$], and the image
$(\www\rho^{-1}\circ \rho)_{hom\to M}$ in $\Aut_M$ 
(defined in theorem \ref{t6.1} (h)) satisfies: For any $s,\www s\in S_\mu$
\begin{eqnarray*}
\www s = (\www\rho^{-1}\circ \rho)_{hom\to M}(s) 
&\iff& (F_s,\pm\rho)\sim_\RR(F_{\www s},\pm\www\rho)\\
&& [\textup{respectively }(F_s,\rho)\sim_\RR(F_{\www s},\www\rho)].
\end{eqnarray*}
\end{corollary}

\begin{proof}
A coordinate change $\varphi\in\RR$ with $F_t=F_{\www t}\circ\varphi$ and
$\rho=\pm \www\rho\circ \varphi_{hom}$ [respectively $\rho=\www\rho\circ\varphi_{hom}$]
exists. Exactly as in the discussion of the homomorphism
$\RR^f\to \Aut_M$ before theorem \ref{t6.1},
it induces an isomorphism $\varphi_M:(M,t)\to (M,\www t)$ of germs of F-manifolds
with Euler fields.
Because $F$ is a very good representative of a semiuniversal unfolding,
$\varphi_M$ is in $\Aut_M$, and $\varphi_M$ lifts to an automorphism 
$(\Phi,\varphi_M)$ of the unfolding. Denote $\Phi_s:= \Phi_{|(\C^{n+1}\times\{s\},0)}$
for $s\in S_\mu$.
Then $f=f\circ \Phi_0$, $F_t=F_{\www t}\circ\Phi_t$, and $\varphi_M = (\Phi_0)_M$.
Thus
$$F_t=F_t\circ\varphi^{-1}\circ\Phi_t,\quad (\varphi^{-1}\circ\Phi_t)_M=\id,$$
and 
$$(\varphi^{-1}\circ\Phi_t)_{hom}=\pm\id$$
by theorem \ref{t6.1} (h)
[$(\varphi^{-1}\circ\Phi_t)_{hom} =\id$ in the case of assumption \eqref{4.1}].
Therefore
$$ f=f\circ\Phi_0,\quad \varphi_M=(\Phi_0)_M,\quad 
\pm\www\rho^{-1}\circ\rho = \varphi_{hom}=\pm(\Phi_0)_{hom}\in G^{mar}_\RR(f)$$
[$(\www\rho^{-1}\circ\rho \in G^{smar}_\RR(f)$ in the case of assumption \eqref{4.1}].
In the case of assumption \eqref{4.2} $-\id\in G^{smar}_\RR(f)$, and also
$\www\rho^{-1}\circ\rho \in G^{smar}_\RR(f)$.

In any case $(\www\rho^{-1}\circ\rho)_{hom\to M}=\varphi_M$. 

Going again through the proof, now with $s$ and $\www s$ instead of $t$ and $\www t$
one obtains $\Leftarrow$. The implication $\Rightarrow$ follows from
$F_s = F_{\varphi_M(s)}\circ\Phi_s$ and $(\Phi_s)_{hom}=\pm \www\rho^{-1}\circ\rho$
[respectively $(\Phi_s)_{hom}=\www\rho^{-1}\circ\rho$ in the case of assumption 
\eqref{4.1}, and with $(F_s,\rho)\sim_\RR (F_s,-\rho)$ in the case of assumption 
\eqref{4.2}].
\end{proof}

\noindent
{\bf Proof of theorem \ref{t4.3}:}
(a) This is clear.

(b) We will use a result of Holmann \cite[Satz 17]{Ho} which shows that
the quotient is an analytic geometric quotient 
if two criteria are satisfied. The first is that the quotient topology is Hausdorff.
The second is the existence of holomorphic functions in a neighborhood of a point
$(f,\pm\rho)$ in $C^{mar}(k,f_0)$ which are constant on $j_k\RR$ orbits and which
separate points in different orbits.

Fix $(f,\pm \rho)\in C^{mar}(k,f_0)$. 
A {\it transversal disk} for $f$ is an embedding $j:\check{M}\to \mmm^2/\mmm^{k+1}$
of an open neighborhood of $0$ in $\C^{\mu-1}$ into $\mmm^2/\mmm^{k+1}$ such that
$j(0)=f$ and $j(\check{M})$ intersects $j_k\RR\cdot f$ transversally in $f$.

By a construction of Gabrielov \cite{Ga} and the result of Teissier \cite[6.14]{Te}
cited above, the germ $(j(\check{M})\cap C(k,f_0),f)$ is isomorphic to the 
$\mu$-constant stratum of $f$ with reduced complex structure 
in a semiuniversal unfolding, by an isomorphism which maps singularities in
$j(\check{M})\cap C(k,f_0)$ to parameters of right equivalent singularities in the
$\mu$-constant stratum
(see \cite[proof of theorem 13.15]{He6} for the details).

Now theorem \ref{t7.2} (b) and corollary \ref{t7.3} show that the quotient topology on
$C^{mar}(k,f_0)/j_k\RR$ is Hausdorff. This gives the first criterion of Holmann.

Let us choose a small submanifold $R\subset j_k\RR$ which contains $\id\in j_k\RR$
and which is transversal at $\id$ to the stabilizer in $j_k\RR$ of $f$.
Then
$$R\times (j(\check{M})\cap C(k,f_0)) \stackrel{\cong}{\longrightarrow}
(\textup{a certain neighborhood of }f\textup{ in }C(k,f_0))=:S(f).$$
The marking $\pm\rho$ in $(f,\pm\rho)$ induces a marking $(g,\pm\rho)$
for all $g\in S(f)$. 
$S(f)\times\{\pm\rho\}$ is a neighborhood of $(f,\pm\rho)$ in $C^{mar}(k,f_0)$.

Because of corollary \ref{t7.3} the $j_k\RR$-orbit of 
$(g,\pm\rho)$ intersects this neighborhood only in
$R\cdot g\times \{\pm\rho\}$. The holomorphic functions on 
$j(\check{M})\cap C(k,f_0)$ lift to this neighborhood and satisfy the 
second criterion of Holmann.

Therefore $M_\mu^{mar}=C^{mar}(k,f_0)/j_k\RR$ is an analytic geometric quotient,
and locally it is isomorphic to the $\mu$-constant stratum of a singularity
with the reduced complex structure.

The canonical complex structures from \cite[theorem 12.4]{He6} on all the $\mu$-constant
strata glue together.
This follows from their construction: By construction, if $(S_\mu,0)$ is a
germ of a $\mu$-constant stratum and $S_\mu$ is a sufficiently small representative
then its canonical complex structure from $(S_\mu,0)$ restricts for any $t\in S_\mu$
to the canonical complex structure on $(S_\mu,t)$. 
Therefore the canonical complex structures
glue to a canonical complex structure on $M_\mu^{mar}$.

(c) The map $\psi_{mar}$ is a bijection, and locally it maps one copy of a
$\mu$-constant stratum of $f$ to another copy, so it is an isomorphism.
The rest is clear.

(d) By definition of $\psi_{mar}$, $[(f_1,\pm \rho_1)]$ and $[(f_2,\pm\rho_2)]\in M_\mu^{mar}$
are in one $G_\Z$-orbit if and only if $f_1\sim_\RR f_2$.
Therefore $M_\mu^{mar}/G_\Z=M_\mu$ as a set.

For some $[(f,\pm\rho)]\in M_\mu^{mar}$ choose a very good representative $F$
of a semiuniversal unfolding with base space $M$ and $\mu$-constant stratum $S_\mu\subset M$
such that a neighborhood of $[(f,\pm\rho)]$ in $M_\mu^{mar}$ is isomorphic to $S_\mu$.

Suppose that $\psi_{mar}([(F_t,\pm\rho)])=[(F_{\www t},\pm\rho)]$ for some $t,\www t\in S_\mu$.
Corollary \ref{t7.3} shows that $\rho^{-1}\circ\psi\circ\rho\in G^{mar}_\RR(f)$,
by theorem \ref{t4.4} (c) then $\psi\in \Stab_{G_\Z}([(f,\pm\rho)])$.
Therefore the action of $G_\Z$ on $M_\mu^{mar}$ is properly discontinuous.

Locally at $[(f,\pm\rho)]$ the quotient is isomorphic to 
$S_\mu/G^{mar}_\RR(f) = S_\mu/\Aut_M$, and this is a neighborhood of $[f]$ in $M_\mu$.
Therefore $M_\mu^{mar}/G_\Z=M_\mu$.

(e) If assumption \eqref{4.1} or \eqref{4.2} holds, the proofs of (b)--(d) can be 
repeated for strongly marked singularities.

Suppose that neither \eqref{4.1} nor \eqref{4.2} holds.
Then an $(f,\rho)\in C^{smar}(k,f_0)$ exists such that $\mult f\geq 3$,
but $\mult g=2$ for arbitrarily close $(g,\rho)$. 
Then $(g,\rho)\sim_\RR (g,-\rho)$, but $(f,\rho)\not\sim_\RR (f,-\rho)$.
The quotient topology of $C^{smar}(k,f_0)/j_k\RR$ does not separate the orbits of 
$(f,\rho)$ and $(f,-\rho)$. So it is not Hausdorff.
This finishes the proof of theorem \ref{t4.3}.  \hfill $\Box$

\section{Examples: Simple and exceptional singularities}\label{c8}
\setcounter{equation}{0}

\noindent
Here we will prove conjecture \ref{t3.2} and conjecture \ref{t5.3}
for the simple singularities and $22$ of the $28$ families of 
exceptional singularities.
Conjecture \ref{t5.3} will use calculations in \cite{He1}
of period maps to $D_{BL}$ (for the exceptional singularities) and
an analysis of conjecture \ref{t3.2} for the simple singularities and
the quasihomogeneous exceptional singularities.

For the remaining $6$ families of exceptional singularities, 
for the simple-elliptic singularities and for 
the hyperbolic singularities the conjectures 
are also true. They will be treated in another paper.

We denote $e(a):=e^{2i\pi a}\in\C$ for $a\in \C$.

\begin{lemma}\label{t8.1}
\cite[lemma 6.5]{He3}
Let $p$ be a prime number, $k,m\in\Z_{\geq 1}$, $c(x)\in\Z[x]$ such that
$c(e(\frac{1}{p^km}))=1$ and $|c(e(\frac{1}{m}))|=1$.

\begin{list}{}{}
\item[(a)] If $p\geq 3$ then $c(e(\frac{1}{m}))=1$.
\item[(b)] If $p=2$ then $c(e(\frac{1}{m}))=\pm 1$.     
\item[(c)] If $p=2$ and $c(e(\frac{1}{p^lm}))=1$ for some $l\in\Z_{\geq 1}-\{k\}$ 
then $c(e(\frac{1}{m}))=1$.     
\end{list}
\end{lemma}

Lemma \ref{t8.2} can be seen as a generalisation of the number theoretic fact: 
{\it For any unit root $\lambda$}
$$\{g(\lambda )\in \Z[\lambda]\ |\ 
  |g(\lambda )|=1\} =\{\pm \lambda^k\ |k\in \Z\} .$$
The proof of lemma \ref{t8.2} uses this fact and lemma \ref{t8.1}.

\begin{lemma}\label{t8.2}
Let $H$ be a free $\Z$-module of finite rank $\mu$, and $H_\C:=H\otimes_\Z\C$.
Let $M_h:H\to H$ be an automorphism of finite order, called monodromy,
with three properties:
\begin{list}{}{}
\item[(i)] 
Each eigenvalue has multiplicity $1$. \\
Denote $H_\lambda:=\ker(M_h-\lambda\cdot\id:H_\C\to H_\C).$
\item[(ii)]
Denote 
$\Ord:=\{\ord\lambda\, |\, \lambda\textup{ eigenvalue of }M_h\}
\subset \Z_{\geq_1}$. 
There exist four sequences $(m_i)_{i=1,...,|\Ord|}$, $(j(i))_{i=2,...,|\Ord|}$, 
$(p_i)_{i=2,...,|\Ord|}$, $(k_i)_{i=2,...,|\Ord|}$ of numbers in $\Z_{\geq 1}$
and two numbers $i_1,i_2\in\Z_{\geq 1}$ with $i_1\leq i_2\leq |\Ord|$ and 
with the properties: \\
$\Ord=\{m_1,...,m_{|\Ord|}\}$,\\ 
$p_i$ is a prime number, $p_i=2$ for $i_1+1\leq i\leq i_2$, $p_i\geq 3$ else, \\
$j(i)=i-1$ for $i_1+1\leq i\leq i_2$, $j(i)<i$ else,
$$m_i=m_{j(i)}/p_i^{k_i}.$$
\item[(iii)]
A cyclic generator $a_1\in H$ exists, that means,
$$H=\bigoplus_{i=0}^{\mu-1}\Z\cdot M_h^i(a_1).$$
\end{list}
Finally, let $I$ be an $M_h$-invariant nondegenerate bilinear form
(not necessarily $(\pm 1)$-symmetric) on $\bigoplus_{\lambda\neq \pm 1}H_\lambda$
with values in $\C$. Then
$$\Aut(H,M_h,I)=\{\pm M_h^k\, |\, k\in \Z\}.$$
\end{lemma}

\begin{proof}
For any $A\in \Aut(H,M_h)$ the polynomial $c(x)=\sum_{i=0}^{\mu-1} c_ix^i\in \Z[x]$
with $A(a_1)=c(M_h)(a_1)$ is well-defined because of (iii). 
Then, also because of (iii), $A=c(M_h)$. 
The eigenvalue of $A$ on $H_\lambda$ is $c(\lambda)$.
It maps $H_{1}$ and $H_{-1}$ to themselves, and $H_{\pm 1}\cap H$ are rank
$1$ sublattices, so $|c(\pm 1)|=1$.

Now suppose $A\in \Aut(H,M_h,I)$. As $I:H_{\lambda}\times H_{\oooo\lambda}\to\C$
for $\lambda\neq\pm 1$ is nondegenerate, $|c(\lambda)|=1$ for such $\lambda$,
hence for all eigenvalues $\lambda$.

By the number theoretic fact cited above, there exist $k\in\Z$ and 
$\varepsilon_1\in\{\pm 1\}$ such that 
$$\varepsilon_1\cdot e(\frac{1}{m_1})^k\cdot c(e(\frac{1}{m_1}))=1.$$
Define $c^{(2)}(x):= \varepsilon_1\cdot x^k\cdot c(x)$, so 
$c^{(2)}(e(\frac{1}{m_1}))=1$. One finds inductively 
$c^{(2)}(e(\frac{1}{m_i}))=1$ for $i=2,...,i_1$ by applying lemma \ref{t8.1} (a)
at each step. Now distinguish two cases.

{\bf Case 1,} $i_1=i_2$, so all $p_i\geq 3$: Define $c^{(3)}(x):=c^{(2)}(x)$.

{\bf Case 2,} $i_1<i_2$: Lemma \ref{t8.1} (b) shows 
$$c^{(2)}(e(\frac{1}{m_{i_1+1}}))=\varepsilon_2\in\{\pm 1\}.$$
Define 
\begin{eqnarray*}
c^{(3)}(x):=\left\{\begin{matrix} c^{(2)}(x) & \textup{ if }\varepsilon_2=1\\
(-x^{m_1/2})\cdot c^{(2)}(x) &\textup{ if }\varepsilon_2=-1.\end{matrix}\right.
\end{eqnarray*}
Then
$$c^{(3)}(e(\frac{1}{m_i}))=1\qquad \textup{for }1\leq i\leq i_1+1.$$
With lemma \ref{t8.1} (c) one finds inductively
$$c^{(3)}(e(\frac{1}{m_i}))=1\qquad\textup{for }i=i_1+2,...,i_2.$$

Now in both cases one finds inductively $c^{(3)}(e(\frac{1}{m_i}))=1$
for $i=i_2+1,...,|\Ord|$, with lemma \ref{t8.1} (a). 
Therefore $c^{(3)}(M_h)=\id$ and  {}\\{} $A\in \{\pm M_h^k\, |\, k\in \Z\}$.
\end{proof}

\begin{theorem}\label{t8.3}
(a) The quasihomogeneous singularities with modality $\leq 2$ and with 
one-dimensional eigenspaces (of the monodromy) are the singularities
$A_\mu$, $D_{2k+1}$, $E_\mu$ and 22 of the 28 quasihomogeneous exceptional
unimodal and bimodal singularities, the exceptions are $Z_{12}$, $Q_{12}$,
$U_{12}$, $Z_{18}$, $Q_{16}$, $U_{16}$.

(b) For all of them 
$$G_\Z :=\Aut(\textup{Milnor lattice, Seifert form}) =\{\pm M_h^k\, |\, k\in \Z\}.$$
This is independent of the number of variables. The orders of the groups
can be read off from the characteristic polynomials (table in the proof).
For the simple singularities they are

\begin{tabular}{llllll}
$A_1$ & $A_\mu\ (\mu\geq 2)$ & $D_\mu\ (\mu=2k+1\geq 5)$ & $E_6$ & $E_7$ & $E_8$\\
$2$ & $2(\mu+1)$ & $4(\mu-1)$ & $24$ & $18$ & $30$
\end{tabular}

(c) For all of them, $M_\mu^{mar}\cong \C^{\mmod(f)}$, here $\mmod(f)\in\{0,1,2\}$
is the modality of $f$, and the period map $M_\mu^{mar}\to D_{BL}$ is an 
isomorphism. Therefore for all of them conjecture \ref{t3.2} (a) 
and conjecture \ref{t5.3} are true (and thus also conjecture \ref{t5.1} 
and conjecture \ref{t5.4}, though conjecture \ref{t5.4}
was shown already in \cite{He1}). 
Also conjecture \ref{t3.2} (b) is true for all of them.

\end{theorem}

\begin{proof}
(a) The following table lists the characteristic polynomials of all
quasihomogeneous surface singularities with modality $\leq 2$.
It can be extracted from the tables of spectral numbers in \cite[13.3.4]{AGV2}
or from \cite{He1}. Inspection of the tables gives (a).
$\Phi_m$ for $m\in \Z_{\geq 1}$ denotes the cyclotomic polynomial of primitive
unit roots of order $m$.

\medskip

\begin{tabular}{ll|ll}
$A_\mu$  & $\frac{t^{\mu+1}-1}{t-1}$ & $E_{3,0}$ & $\Phi_{18}^2\Phi_6\Phi_2^2$\\
$D_\mu$  & $(t^{\mu-1}+1)\Phi_2$     & $Z_{1,0}$ & $\Phi_{14}^2\Phi_2^3$\\
$E_6$    & $\Phi_{12}\Phi_3 $        & $Q_{2,0}$ & $\Phi_{12}^2\Phi_4^2\Phi_3$\\
$E_7$    & $\Phi_{18}\Phi_2 $        & $W_{1,0}$ & $\Phi_{12}^2\Phi_6\Phi_4\Phi_3\Phi_2$\\
$E_8$    & $\Phi_{30}$               & $S_{1,0}$ & $\Phi_{10}^2\Phi_5\Phi_2^2$\\
$\www E_6$&$\Phi_3^3\Phi_1^2$        & $U_{1,0}$ & $\Phi_9^2\Phi_3$\\
$\www E_7$&$\Phi_4^2\Phi_2^3\Phi_1^2$&           & \\
$\www E_8$&$\Phi_6\Phi_3^2\Phi_2^2\Phi_1^2$&     & 
\end{tabular}

\begin{tabular}{ll|ll}
$E_{12}$ & $\Phi_{42}$               & $E_{18}$  & $\Phi_{30}\Phi_{15}\Phi_3$\\
$E_{13}$ & $\Phi_{30}\Phi_{10}\Phi_2$    & $E_{19}$ & $\Phi_{42}\Phi_{14}\Phi_2$\\
$E_{14}$ & $\Phi_{24}\Phi_{12}\Phi_3$   & $E_{20}$ & $\Phi_{66}$\\
$Z_{11}$ & $\Phi_{30}\Phi_6\Phi_2$      & $Z_{17}$ & $\Phi_{24}\Phi_{12}\Phi_6\Phi_3\Phi_2$\\
$Z_{12}$ & $\Phi_{22}\Phi_2^2$          & $Z_{18}$ & $\Phi_{34}\Phi_2^2$\\
$Z_{13}$ & $\Phi_{18}\Phi_9\Phi_2$      & $Z_{19}$ & $\Phi_{54}\Phi_2$\\
$Q_{10}$ & $\Phi_{24}\Phi_3$            & $Q_{16}$ & $\Phi_{21}\Phi_3^2$\\
$Q_{11}$ & $\Phi_{18}\Phi_{6}\Phi_{3}\Phi_{2}$ & $Q_{17}$ &   $\Phi_{30}\Phi_{10}\Phi_{6}\Phi_{3}\Phi_2$\\
$Q_{12}$ & $\Phi_{15}\Phi_{3}^2$        & $Q_{18}$ & $\Phi_{48}\Phi_{3}$\\
$W_{12}$ & $\Phi_{20}\Phi_{5}$          & $W_{17}$ & $\Phi_{20}\Phi_{10}\Phi_{5}\Phi_{2}$\\
$W_{13}$ & $\Phi_{16}\Phi_{8}\Phi_{2}$  & $W_{18}$ & $\Phi_{28}\Phi_{7}$\\
$S_{11}$ & $\Phi_{16}\Phi_{4}\Phi_{2}$  & $S_{16}$ & $\Phi_{17}$\\
$S_{12}$ & $\Phi_{13}$                  & $S_{17}$ & $\Phi_{24}\Phi_{8}\Phi_{6}\Phi_{3}\Phi_2$\\
$U_{12}$ & $\Phi_{12}\Phi_{6}\Phi_{4}^2\Phi_{2}^2$ & $U_{16}$ & $\Phi_{15}\Phi_{5}^2$
\end{tabular}

\medskip

(b) In section \ref{c2} $G_\Z(f)=G_\Z(f+x_{n+1}^2)$ was shown. 
Therefore it is sufficient to show (b) for the surface singularities in (a).
Lemma \ref{t8.2} with $H=Ml(f)$, $I=$ {\it intersection form} or {\it Seifert form},
$M_h=$ {\it monodromy} shall be applied to the surface singularities in (a).
Condition (i) is clear. Condition (ii) can be checked by inspection of the
table of characteristic polynomials above (only for $D_{2k+1}$, $Q_{11}$ and
$Q_{17}$ one has to choose $i_1>1$). 

Condition (iii) is a special case of the following conjecture of Orlik \cite{Or}:

{\it For a quasihomogeneous singularity consider the unique decomposition
of its characteristic polynomial $p_{ch}$ into a product 
$p_{ch}=p_1\cdot ...\cdot p_l$ of unitary polynomials with $p_l|p_{l-1}|...|p_1$,
$p_l\neq 1$. Then $Ml(f)$ is a direct sum of cyclic modules, 
$$Ml(f)=\bigoplus_{i=1}^l\left(\bigoplus_{j=1}^{\deg p_j}\Z\cdot M_h^{j-1}(a_j)\right)$$
for suitable $a_1,...,a_l\in Ml(f)$ such that the monodromy on the $j$-th block
has characteristic polynomial $p_j$.}

Of course, if the conjecture holds for a singularity $f(x_0,...,x_n)$ then
it holds also for any suspension $f(x_0,...,x_n)+x_{n+1}^2$. 
Michel and Weber can prove the conjecture for $n=1$ \cite{MW}.
In \cite[3.1]{He1} it is proved (using Coxeter-Dynkin diagrams) for those 
quasihomogeneous surface singularities with modality $\leq 2$ 
which are not suspensions of curve singularities. So, for the singularities
in (a) the conjecture is true with $l=1$. There condition (iii) holds.
Lemma \ref{t8.2} applies and gives the statement. 

(c) By theorem \ref{t3.3} (c) (and theorem \ref{t4.4} (c)), 
for any quasihomogeneous singularity
$$\{\pm M_h^k\, |\, k\in\Z\}\subset \Stab_{G_\Z(f)}([(f,\pm \id)])
=G^{mar}_\RR(f)\subset G^{mar}(f)\subset G_\Z(f).$$
Part (b) gives equalities for the singularities $f$ in (a), 
especially $G^{mar}(f)=G_\Z(f)$, which is conjecture \ref{t3.2} (a).
By theorem \ref{t4.4} (a) $M_\mu^{mar}$ is connected.

In \cite{AGV1} for the simple and the exceptional singularities, 
holomorphic $\mu$-constant families with base spaces $X\cong \C^{\textup{mod}(f)}$
are given. The base space is equipped
with a good $\C^*$-action (good = positive weights), the point $0$ stands
for the quasihomogeneous singularity, the other points for semiquasihomogeneous
singularities. 

For any singularity $\www f_0$, the moduli space $M_\mu^{mar}(\www f_0)$ 
comes equipped with a $\C$-action, by
$$\C\times M_\mu^{mar}(\www f_0),\quad (c,[(\www f,\pm\rho)])
\mapsto [(e^{-c}\cdot \www f,\pm\rho\circ\sigma(c)],$$
here $\sigma(c):ML(e^{-c}\cdot\www f)\to ML(\www f)$ is the canonical
isomorphism within the $\mu$-constant family
$\{e^{-c}\cdot \www f\, |\, c\in\C\}$.

In the case of the quasihomogeneous exceptional singularities $f$, 
the germ $(M_\mu^{mar}(f),[(f,\pm \id)])$ and the germ $(X,0)$ are isomorphic.
The $\C$-action on $M_\mu^{mar}(f)$ factors through to an action of 
$\C/2\pi im\Z\cong \C^*$ where $m\in\Z_{\geq 1}$ is minimal with 
$M_h^m=\pm\id$. Any class in $M_\mu^{mar}(f)$ is obtained by this 
$\C^*$-action from a class close to $[(f,\pm\id)]$. 
The isomorphism of germs $(M_\mu^{mar}(f),[(f,\pm\id)])\cong (X,0)$
is compatible with the $\C^*$-actions. 
Therefore $$M_\mu^{mar}(f)\cong X.$$

In \cite{He1} (copied in \cite{He2} and \cite{Ku}) the period map 
$X\to D_{BL}$ was calculated and shown to be an isomorphism.
Therefore conjecture \ref{t5.3} is true
(and thus also conjecture \ref{t5.1} and conjecture \ref{t5.4}, 
though conjecture \ref{t5.4} was shown already in \cite{He1}). 

Finally we come to $G^{smar}$ and $M_\mu^{smar}$ and conjecture \ref{t3.2} (b).
Assumption \eqref{4.1} or \eqref{4.2} holds. Suppose that assumption \eqref{4.1}
holds, that means that we consider curve or surface singularities, depending
on the type. Then by theorem \ref{t4.4} (e), $\{\pm\id\}\subset G_\Z$ acts  freely
on $M_\mu^{smar}$ with quotient $M_\mu^{mar}$, and the quotient map 
is a double covering. But $M_\mu^{mar}\cong \C^{\mmod(f)}$. The only possible
double covering is that which maps two copies of $M_\mu^{mar}$ to $M_\mu^{mar}$.
By theorem \ref{t4.4} (b) then $G_\Z=G^{smar}(f)\times\{\pm\id\}$
and 
$$G^{smar}(f)=\Stab_{G_\Z}([(f,\id)])=\{M_h^k\, |\, k\in\Z\}$$
(here $M_h$ is the monodromy of $f$ with $\mult(f)\geq 3$; whether this is
the surface or curve singularity, depends on the type.).
\end{proof}

\begin{theorem}\label{t8.4}
Also for the simple singularities $D_{2k}$
the conjectures \ref{t3.2} (a) and (b) are true.
Therefore here $M_\mu^{mar}\cong \{pt\}$, and 
the period map $M_\mu^{mar}\to D_{BL}$ is an 
isomorphism (mapping one point to one point). 
Therefore the conjectures \ref{t5.3} and \ref{t5.1} are true
(conjecture \ref{t5.4} is trivial for the simple singularities.)

In the case $mult f\geq 3$ (that is, $\mult f=3$, it is the case $n=1$)
\begin{eqnarray*}
G^{smar}&=&\{M_h^a\, |\, a=0,1,...,2k-2\}\times U,\\
G_\Z&=&G^{mar}=G^{smar}\times \{\pm\id\}
\end{eqnarray*}
with 
$U\cong S_3$ for $D_4$ and $U\cong S_2$ for $k\geq 3$.
So, $|G_\Z|=36$ for $D_4$ and $|G_\Z|=4(\mu-1)$ if $\mu=2k\geq 6$.
\end{theorem}

\begin{proof}
Consider the curve singularity $f=x^{2k-1}-xy^2$ with weights
$(w_x,w_y)=(\frac{1}{2k-1},\frac{k-1}{2k-1})$. Because of $\mmod(f)=0$
and theorem \ref{t4.4} (b), 
$M_\mu^{smar}$ consists of $|G_\Z/G^{smar}|$ many points, and 
$G^{smar}=\Stab_{G_\Z}([(f,\id)])$. By theorem \ref{t6.2}
and theorem \ref{t6.1} (g), this group can be calculated:
The restriction of the homomorphism $()_{hom}:\RR^f\to \Stab_{G_\Z}([(f,\id)])$
to the finite group $\Stab_{G_w}(f)\subset \RR^f$ with 
$$\Stab_{G_w}(f):=\{\varphi\in\RR^f\, |\,  \varphi\textup{ is a quasihomogeneous
coordinate change}\}$$
is an isomorphism $()_{hom}:\Stab_{G_w}(f)\to \Stab_{G_\Z}([(f,\id)])$.

It is easy to see that $\Stab_{G_w}(f)$ is generated by the coordinate changes
\begin{eqnarray*}
\varphi_1:(x,y)&\mapsto& (e(w_x)x,e(w_y)y)\qquad \textup{ with }(\varphi_1)_{hom}=M_h,\\
\varphi_2:(x,y)&\mapsto& (x,-y),\\
\varphi_3:(x,y)&\mapsto& (-\frac{1}{2}x+\frac{1}{2}y,\frac{3}{2}x+\frac{1}{2}y)
\quad\textup{ only for }k=2.
\end{eqnarray*}
The element $\varphi_2$ [and $\varphi_3$ if $k=2$] 
generates a subgroup of $G_w$ isomorphic to $S_2$
[respectively $S_3$ if $k=2$].
The image under $()_{hom}$ is called $U$. Thus
$$G^{smar}=\Stab_{G_\Z}([(f,\id)])=\{M_h^k\, |\, k=0,1,...,2k-2\}\times U.$$
By theorem \ref{t3.3} (g) $-\id\notin \Stab_{G_\Z}([(f,\id)])$,
therefore $G^{mar}=G^{smar}\times\{\pm\id\}$
and $|G^{mar}|=36$ for $D_4$ and $|G^{mar}|=4(\mu-1)$ for $k\geq 3$.
It rests to see $|G_\Z|=|G^{mar}|$. Then $G_\Z=G^{mar}$, and everything is proved.

For the calculation of $|G_\Z|$ we go over to $\www f=f(x,y)+x_2^2+x_3^2+x_4^2$ with 
$n=4$. It is well known that then $(Ml(\www f),I)$ is the root lattice of type
$D_{2k}$ and $M_h$ is a Coxeter element. Then $G_\Z=\Aut(Ml(f),M_h,I)$ because
$I$ is nondegenerate.

Choose a basis $e_1,...,e_\mu$ of $Ml(\www f)$ which corresponds to the standard
Dynkin diagram,
\begin{eqnarray*}
&& I(e_i,e_i)=2,\quad I(e_i,e_{i+1})=-1\textup{ for }1\leq i\leq \mu-2,\quad
I(e_{\mu-2},e_\mu)=-1.\\
&& I(e_i,e_j)=0\quad \textup{ for all other }i\textup{ and }j\textup{ with }i<j.
\end{eqnarray*}
Then
\begin{eqnarray*}
M_h &=& s_{e_1}\circ s_{e_2}\circ ...\circ s_{e_\mu},\\
B_1&:=&\ker(M_h+\id)\cap Ml(\www f) = \Z\cdot b_1 + \Z\cdot b_2,\quad where\\
 b_1&=&e_{\mu-1}-e_\mu,\quad b_2=e_{\mu-1}+e_{\mu-3}+...+e_1,\\
&& I(b_1,b_1)=4,\quad I(b_1,b_2)=2,\quad I(b_2,b_2)=2k.
\end{eqnarray*}
If $k=2$ this is an $A_2$-lattice.
If $k\geq 3$, $I(\beta,\beta)=4\iff \beta=\pm b_1$ for $\beta\in B_1$.
One sees easily $|\Aut(B_1,I)|=12$ if $k=2$ and 
$|\Aut(B_1,I)|=4$ if $k\geq 3$.

Any $\psi\in G_\Z$ maps $B_1$ to itself because $\psi\circ M_h=M_h\circ\psi$.
There is an exact sequence
$$1\mapsto \{\psi\in G_\Z\, |\, \psi=\id\textup{ on }B_1\}
\to G_\Z\to \Aut(B_1,I).$$
Orlik's conjecture holds also for $D_{2k}$ (see the proof of theorem \ref{t8.3} (b)),
with $p_{ch}=(t^{\mu-1}+1)\Phi_2 = p_1\cdot p_2$, $p_1=t^{\mu-1}+1$, $p_2=\Phi_2$:
There are $a_1\in Ml(\www f)$, $a_2\in B_1$ with 
$$Ml(\www f)= \left(\bigoplus_{i=0}^{\mu-2}\Z\cdot M_h^i(a_1)\right) \oplus \Z\cdot a_2
=: B_2\oplus \Z\cdot a_2.$$
Any $\psi\in G_\Z$ with $\psi =\id$ on $B_1$ restricts to an 
automorphism of $B_2$. Lemma \ref{t8.2} applies and shows
$\psi_{|B_2}= \pm (M_{h|B_2})^k$ for some $k$. Now $\psi_{|B_1}=\id$ forces
$\psi=(-M_h)^k$ for some $k\in\Z$ (here $M_h$ is the mondromy of $\www f$,
so minus the monodromy of $f$). Now
\begin{eqnarray*}
|G_\Z| &=& |\{(-M_h)^k\, |\, k\in\Z\}|\cdot |\textup{image of }G_\Z\textup{ in }
\Aut(B_1,I)|\\
&\leq&\left\{ \begin{matrix} 3\times 12 = |G^{mar}| & \textup{ if }k=2,\\
(2k-1)\cdot 4 =|G^{mar}| & \textup{ if }k\geq 3.\end{matrix}\right.
\end{eqnarray*}
Therefore
$G_\Z=G^{mar}$. And $G_\Z\to \Aut(B_1,I)$ is surjective.
\end{proof}

\begin{remarks}\label{t8.5}
(i) Let $f$ be any quasihomogeneous singularity such that its Milnor lattice and
its monodromy satisfy the properties (i) -- (iii) in lemma \ref{t8.2}.
Property (i) implies that there are no nontrivial $\mu$-constant deformations
of weight $0$. Therefore a holomorphic $\mu$-constant family with base space 
$X\cong \C^{\textup{mod}(f)}$ and good $\C^*$-action on it exists,
where the point $0$ stands for the quasihomogeneous singularity and all other
points for semiquasihomogeneous singularities, and it contains representatives
of any right equivalence class in the $\mu$-homotopy class of $f$ \cite{AGV1}.
Then the proof of theorem \ref{t8.3} (c) goes through without changes.
Then $M_\mu^{mar}\cong X\cong \C^{\textup{mod}(f)}$, and the period map
$M_\mu^{mar}\to D_{BL}$ is an isomorphism.

In \cite[ch. 6]{He3} most of this (only not conjecture \ref{t3.2} (b) for $G^{smar}$)
was carried out for the Brieskorn-Pham singularities with pairwise
coprime exponents.

(ii) In theorem \ref{t8.3} the quasihomogeneous singularities
$Z_{12}$, $Q_{12}$, $U_{12}$, $Z_{18}$, $Q_{16}$, $U_{16}$ are missing.
Also they have holomorphic $\mu$-constant families with base spaces $X\cong\C^{\mmod(f)}$
as in (i). A part of the proof of theorem \ref{t8.3} applies and shows
that each component of $M_\mu^{mar}$ is isomorphic to $X$, and it shows that
in the case $\mult(f)\geq 3$ conjecture \ref{t3.2} (b) is true.
The only missing part is a proof of conjecture \ref{t3.2} (a): $G^{mar}=G_\Z$.
This will be shown in another paper.
Then the proof of theorem \ref{t8.3} goes through and gives everything else.
Then the period map $M_\mu^{mar}\to D_{BL}$ is an isomorphism.
%

(iii) For the simple singularities $G_\Z$ had also been calculated in \cite{Yu1}\cite{Yu2}.
There specific properties of the ADE root lattices were used.
Also $G_\Z=G^{mar}$ is contained implicitly in \cite{Yu1}\cite{Yu2}.

(iv) In the case of the simple and the exceptional singularities, $M_\mu^{mar}$
[respectively any component of $M_\mu^{mar}$ for the cases 
$Z_{12}$, $Q_{12}$, $U_{12}$, $Z_{18}$, $Q_{16}$, $U_{16}$] is simply connected.
A discussion in another paper will show that this holds also for the
simple-elliptic singularities $\www E_k$ ($k=6,7,8$) and the hyperbolic
singularities $T_{pqr}$ ($\frac{1}{p}+\frac{1}{q}+\frac{1}{r}<1)$.

But I expect that it does not hold for the six bimodal quadrangle singularities.
In those cases $D_{BL}$ is a line bundle over the hyperbolic plane $\H$, and
the image of the period map $M_\mu^{mar}\to D_{BL}$ is the restriction of the line bundle
to the complement in $\H$ of the countably many elliptic fixed points of a certain
triangle group  with angles $>0$ \cite{He1}.
\end{remarks}

\end{document}